\documentclass[12pt]{amsart}
\usepackage{amsmath,amssymb,amsfonts,amsthm,amsopn,dsfont}
\usepackage[dvips]{graphicx}
\usepackage{srcltx}
\usepackage{setspace}
\newtheorem{theorem}{Theorem}[section]

\newtheorem{definition}[theorem]{Definition}

\newtheorem{lemma}[theorem]{Lemma}
\numberwithin{equation}{section}

\newtheorem{proposition}[theorem]{Proposition}%%

\renewcommand{\ell}{l}
\renewcommand{\epsilon}{\varepsilon}
\def\al{\alpha}                    % Greek Letters
\def\be{\beta}
\def\ga{\gamma}

\def\ve{\varepsilon}
\def\eps{\varepsilon}

 \def\cS{\mathcal{S}}

\def\px{\langle x \rangle}
\def\pxi{\langle \xi \rangle}

\def\rd{\bR^d}

\def\rdd{{\bR^{2d}}}

\def\R{\right)}

\def\<{\left<}
\def\>{\right>}

\def\mv1{M_v^1}

\def\mn{(m,n)}
\def\mn'{(m',n')}

\newcommand{\eab}{\eps^{|\alpha|+|\beta|}}
\newcommand{\mab}{M(\alpha,\beta)}
\newcommand{\tg}{\tilde{\gamma}}
\newcommand{\tid}{\tilde{\delta}}
\newcommand{\sg}{\mathcal{H}_{ sect}(\rd)}

\def\N{\mathbb{N}}
\def\R{\mathbb{R}}
\def\C{\mathbb{C}}
\def\rd{\mathbb{R}^d}
\def\rdd{\mathbb{R}^{2d}}

\begin{document}

\begin{abstract}
We prove sharp analytic regularity and decay at infinity of
solutions of variable coefficients nonlinear harmonic
oscillators. Namely, we show holomorphic extension to a
sector in the complex domain, with a corresponding Gaussian
decay, according to the basic properties of the Hermite
functions in $\rd$. Our results apply, in particular, to
nonlinear eigenvalue problems for the harmonic oscillator
associated to a real-analytic scattering, or asymptotically
conic, metric in $\rd$, as well as to certain perturbations
of the classical harmonic oscillator.

\end{abstract}

\title[Nonlinear harmonic oscillators]{Regularity and decay of solutions of nonlinear
harmonic oscillators}
\author{Marco Cappiello \and Fabio Nicola}
%    Address of record for the research reported here
\address{Dipartimento di Matematica,  Universit\`{a} degli Studi di Torino,
Via Carlo Alberto 10, 10123
Torino, Italy}
\address{Dipartimento di Matematica, Politecnico di
Torino, Corso Duca degli
Abruzzi 24, 10129 Torino,
Italy}
%    Current address
%\curraddr{}
\email{marco.cappiello@unito.it}
\email{fabio.nicola@polito.it}
\date{}

\maketitle
%\footnotetext[0]{This paper is published despite the effects of the Italian law 133/08. This law drastically reduces public funds to public Italian universities, which is particularly dangerous for scientific free research, and it will prevent young researchers from getting a position, either temporary or tenured, in Italy. The authors are protesting against this law to obtain its cancellation. See http://groups.google.it/group/scienceaction.}

\section{Introduction}
The harmonic oscillator $H=-\Delta+|x|^2$ in ${\rd}$
represents one of the simplest and yet more useful models
for several physical phenomena, and its relevance both in
Mathematical Analysis and Physics is well-known. Its
eigenfunctions, namely the Hermite functions $h_\alpha(x)$,
are given by the formulae
$h_\alpha(x)=p_\alpha(x)e^{-|x|^2/2}$, $\alpha\in \N^d$,
where $p_\alpha$ is a polynomial of degree $|\alpha|$ (see
e.g. \cite{thangavelu}). Two remarkable features of the
Hermite functions are their Gaussian decay at infinity, and
their very high regularity. In fact, we have
\begin{equation}\label{intro}
|h_\alpha(x)|\lesssim e^{- c |x|^2}\quad{\rm for}\ x\in\rd,\qquad |\widehat{h_\alpha}(\xi)|\lesssim  e^{- c |\xi|^2}\quad{\rm for}\ \xi\in\rd
\end{equation}
for every $c<1/2$, where $\widehat{h_\alpha}(\xi)$ denotes
the Fourier transform of $h_\alpha$. The functions $h_\alpha$ in fact extend to entire functions $h_\alpha(x+iy)$ in the complex space $\C^d$ and, for every $0<\epsilon<1$, we have the estimates
\begin{equation}\label{intro12ma}
|h_\alpha(x+iy)|\lesssim e^{- c |x|^2}\quad {\rm in\ the\
sector}\,\ |y|<\epsilon (1+|x|),
\end{equation}
 for some $c>0$.\par
In this paper we wonder to what extent these properties
continue to hold for {\it nonlinear} perturbations of the
harmonic oscillator, possibly with variable coefficients.
Relevant models  are equations of the type
\begin{equation}\label{12ma}
-\Delta u+|x|^2 u-\lambda u=F[u],\qquad \lambda\in\C,
\end{equation}
with a nonlinearity of the form $F[u]=\sum_{|\alpha|+|\beta|\leq 1} c_{\alpha\beta} x^\beta\partial^\alpha u^k$, $k\geq 2$. Cappiello, Gramchev and Rodino in \cite{A18} showed by a counterexample that generally, even in dimension $d=1$, there can exist Schwartz solutions of \eqref{12ma} which do not extend to entire functions in $\C$. In fact, a refinement of their argument (see Section \ref{esempi} below) shows that a sequence of complex singularities may occur, approaching a straight line at infinity.  On the other hand, as a positive result, it was proved in \cite{A18} that every solution $u\in H^s(\rd)$, $s>d/2+1$, of \eqref{12ma} extends to a holomorphic function $u(x+iy)$ on the strip $\{z \in \C^d: |{\rm Im}\, z|<T\}$ and satisfies there an estimate of the type $|u(x+iy)|\leq Ce^{-c|x|^2}$, for some $c,C,T>0$. Similar results, namely, holomorphic extension to a {\it strip} and super-exponential decay, were proved in \cite{A18, CGR2} for more general classes of elliptic operators with polynomial coefficients. %Generalizations to classes of operators with polynomial coefficients including models as $-\Delta+|x|^{2k}, k \in \N$, were proved in \cite{CGR2}.
\par
The above mentioned negative result as well as the
estimates \eqref{intro12ma}, valid in a sector in the
linear case, suggest the possibility, even in the presence
of certain nonlinear perturbations, of a holomorphic
extension of the solutions to a {\it sector}, rather than
only a strip, with a corresponding Gaussian decay estimate.
In this paper we show, for a large class of equations
including \eqref{12ma}, even with {\it non-polynomial}
coefficients, that this is in fact the case. The techniques
developed here actually will apply to much more general
dif\-fe\-ren\-tial (and pseudodifferential) operators. To
motivate the class of operators we will consider, we first
discuss a special yet important example.
\par\medskip\noindent Consider the equation $Pu=F[u],$ with
\begin{equation}\label{intro1}
P=\sum_{j,k=1}^d g^{jk}(x)\partial_j\partial_k +\sum_{k=1}^d b_k(x)\partial_k +V(x),
\end{equation}
where the functions $g^{jk}$, $b_k$, and the potential $V$ are real-analytic in $\rd$, and satisfy the following conditions.\par We suppose that the matrix $\big(g^{jk}\big)$ is real and symmetric and that there exists a constant $C>0$ such that
\begin{equation}\label{intro2}
\sum_{j,k=1}^dg^{jk}(x)\xi_j\xi_k\geq C^{-1}|\xi|^2\qquad{\forall x,\xi\in\rd},
\end{equation}
as well as
\begin{equation}\label{intro3}
|\partial^\alpha g^{jk}(x)|+|\partial^\alpha b_k(x)|\leq C^{|\alpha|+1}\alpha!\langle x\rangle^{-|\alpha|}\quad\forall x\in\rd, \alpha\in\N^d,
\end{equation}
where $\px=(1+|x|^2)^{1/2}.$
Moreover we assume that
\begin{equation}\label{intro4} \begin{cases}
{\rm Re}\,V(x)\geq C^{-1} |x|^2\qquad \qquad {\rm for}\ |x|>C,\\ |\partial^\alpha V(x)|\leq C^{|\alpha|+1}\alpha!\langle x\rangle^{2-|\alpha|} \quad\forall x\in\rd,\ \alpha\in\N^d \end{cases}.
\end{equation}
We consider a nonlinearity of the form
\begin{equation}\label{intro5}
F[u]=\sum_{2\leq h+l\leq N\atop 1\leq j\leq d} F_{jhl}(x)u^h(\partial_j u)^l,
\end{equation}
for some $N\in\N$, where
\begin{equation}\label{intro6}
|\partial^\alpha F_{jhl}(x)|\leq C^{|\alpha|+1}\alpha!\langle x\rangle^{1-\min\{1,l\}-|\alpha|} \quad\forall x\in\rd,\ \alpha\in\N^d.
\end{equation}
Then, we claim that\par\smallskip {\it Under these
assumptions, every solution $u \in H^s(\R^d), s>d/2+\min
\{1,l\}$, of the equation $Pu=F[u]$, extends to a
holomorphic function $u(x+iy)$ in the sector $
\{z=x+iy\in\mathbb{C}^d:\ |y|<\eps(1+|x|)\} $ of
$\mathbb{C}^d$ for some $\eps>0$, satisfying there the
estimates $ |u(x+iy)|\leq Ce^{-c|x|^2}, $ for some
constants $C>0$, $c>0$.}\par\smallskip Notice that if
$V(x)$ satisfies \eqref{intro4} then also $V(x)-\lambda,
\lambda \in \C,$ satisfies it, so that the above result
applies to the corresponding eigenvalue problem as well. In
the linear case ($F[u]=0$) this result intersects the wide
literature on the decay and regularity of eigenfunctions of
Schr\"odinger operators, cf. Agmon \cite{Ag}, Nakamura
\cite{Na}, Sordoni \cite{So}, Rabinovich \cite{Ra},
Rabinovich and Roch \cite{RR} and many others.\par We also
remark that suitable perturbations of the standard harmonic
oscillator fall in this class of equations, as well as the
harmonic oscillator associated to a real-analytic
scattering, or asymptotically conic, Riemannian metric in
$\rd$ (see Section \ref{esempi} below). For a detailed
analysis of these metrics and their important role in
geometric scattering theory we refer to Melrose
\cite{melrose1, melrose2}, Melrose and Zworski
\cite{melrose3}.\par

 \par\medskip\noindent
Let us now state our main result in full generality. We
consider nonlinear equations whose linear part is a
differential or even pseudodifferential operator
\begin{equation}\label{pr0}
Pu(x)=p(x,D)u(x)=(2\pi)^{-d}\int_{\R^d} e^{ix\xi}
p(x,\xi)\widehat{u}(\xi)\,d\xi,
\end{equation}
with symbol $p$ in the class $\Gamma_a^m(\R^d), m>0,$ defined as the space of all functions $p \in C^{\infty}(\R^{2d})$ satisfying the estimates
\begin{equation}
\label{Gsymbols}
|\partial_{\xi}^{\alpha} \partial_x^{\beta}p(x,\xi)|\leq C^{|\alpha|+|\beta|+1}\alpha!\beta!(1+|x|+|\xi|)^{m-|\alpha|}\px^{-|\beta|}
\end{equation}
for all $(x,\xi) \in \R^{2d}, \alpha, \beta \in \N^d,$ and for some positive constant $C$
independent of $\alpha, \beta.$ This class is particularly suited to study harmonic oscillators
 with variable analytic coefficients and it is inspired by the class considered
 by Shubin  \cite{Shubin:1} and Helffer \cite{Helffer:1} which was in fact
 modelled on the harmonic oscillator and its real powers. However, a {\it differential} operator belongs to that class only if its coefficients are polynomial, which is an unpleasant limitation. With respect to \cite{Helffer:1}, \cite{Shubin:1}, we overcome this restriction by assuming the less demanding estimates \eqref{Gsymbols}, which at the same time imply that the symbol $p$ is  real-analytic. For example, a function $$p(x,\xi)= \sum_{|\alpha|\leq m}c_{\alpha}(x)\xi^{\alpha}$$ belongs to $\Gamma_a^m(\R^{d})$ if the coefficients $c_{\alpha}$ satisfy
$|\partial^\beta c_\alpha(x)|\leq C^{|\beta|+1}\beta!\langle x\rangle^{m-|\alpha|-|\beta|}$.
\\We shall assume moreover the symbol $p$ of our operator to be $\Gamma$-elliptic, in the sense that, for some constant $R>0$,
\begin{equation}\label{Gammaell}
\inf_{|x|+|\xi| \geq R}(1+|x|+|\xi|)^{-m} |p(x,\xi)|>0.
\end{equation}
This is clearly a global version of the classical notion of ellipticity. For example,
 by \eqref{intro2}--\eqref{intro4}, the symbol of the operator in \eqref{intro1} belongs to $\Gamma_a^2(\R^{d})$ and satisfies \eqref{Gammaell} with $m=2$.

We moreover consider a nonlinearity of the form
\begin{equation}\label{A5.2}
F[u]= \sum_{h,l,\rho_{1},\ldots, \rho_{l}}F_{h,l,\rho_{1}\ldots \rho_{l}}(x) \prod_{k=1}^l\partial^{\rho_{k}}u ,
\end{equation}
where the above sum is finite and $h,l \in \N, l \geq 2, \rho_{1},\ldots, \rho_{l}\in \N^{d}$ satisfy the condition $h+\max\{|\rho_{k}|\} \leq \max\{m-1,0\}$. We assume that the functions $F_{h,l,\rho_{1}\ldots \rho_{l}}(x) $ satisfy the following estimates
\begin{equation}
\label{estnonlincoeff}
|\partial^{\beta}F_{h,l,\rho_{1}\ldots \rho_{l}}(x)|\leq C^{|\beta|+1}\beta! \px^{h-|\beta|},
\end{equation}
for some positive constant $C$ depending on $h,l,\rho_1,\ldots, \rho_{l}$ and independent of $\beta.$ Our main result is the following.

\begin{theorem}\label{mainthm}
Let $p \in \Gamma_a^m(\R^{d}), m>0,$ satisfy
\eqref{Gammaell} and let $F[u]$ be of the form
\eqref{A5.2}, \eqref{estnonlincoeff} (possibly with some
factors in the product replaced by their conjugates).
Assume that $u \in H^s(\R^d), s>d/2+\max_k \{|\rho_k|\}$ is
a solution of the equation $Pu=F[u]$. Then $u$ extends to a
holomorphic function $u(x+iy)$ in the sector
\begin{equation}\label{cla3intro}
\{z=x+iy\in\mathbb{C}^d:\ |y|<\eps(1+|x|)\}
\end{equation}
of $\mathbb{C}^d$, for some $\eps>0$, satisfying there the estimates
\begin{equation}\label{cla4intro}
|u(x+iy)|\leq Ce^{-c|x|^2},
\end{equation}
for some constants $C>0$, $c>0$. The same holds for all the derivatives of $u$.
\end{theorem}

\indent The linear case $F[u]=0$ deserves a special interest and will be treated in detail in Section \ref{esempi}. We emphasize the fact that the form of the domain of holomorphic extension as a sector is, in a sense, completely sharp, even for the model \eqref{12ma} (see Section \ref{esempi}). \par
Let us briefly compare our result with those in the existent literature. Several papers were devoted to the problem of holomorphic extension to a strip and exponential decay of solutions of certain semilinear elliptic equations arising in the theory of solitary waves or bound states, whose model is $-\Delta u+u=|u|^{p-1}u$, cf. Beresticky and Lions \cite{BeLi}, Bona and Grujic' \cite{bo1}, Bona and Li \cite{BL}, \cite{BL2}, Biagioni and Gramchev \cite{BG}, Gramchev \cite{Gr}, Cappiello, Gramchev and Rodino \cite{A39}, \cite{CGR} and the references therein; see also our recent paper \cite{cappiello-nicola} for the extension to a sector. However, as it is clear from our model \eqref{12ma}, we consider a different class of equations here, and in fact we deal with Gaussian, rather than exponential decay. Instead, as already mentioned, a class similar to the present one was considered in \cite{A18}, where the problem of the extension to a strip, combined with super-exponential decay, was addressed. The main novelties of the present work are the possibility of treating non-polynomial coefficients and nonlinearities, and the achievement of the optimal extension result, namely to a sector. \\
\indent The paper is organized as follows. In section
\ref{prelimi} we list some known factorial and binomial
estimates and we collect some basic properties of the
pseudodifferential operators introduced before, which will
be instrumental in the proofs of our results. In Section
\ref{secspazi} we introduce a suitable space of analytic
functions which exploits the two properties
\eqref{cla3intro} and \eqref{cla4intro}. Section
\ref{secdim} is devoted to the proof of Theorem
\ref{mainthm} which is based on an iterative scheme on the
space defined in Section \ref{secspazi}. Finally, in
Section \ref{esempi} we give some concluding remarks. In
particular, we read our results on the models introduced
above and treat in detail their application to the
Schr\"odinger operator in $\rd$ with a scattering metric.
Finally, we discuss the sharpness of our results for what
concerns the shape of the domain of the holomorphic
extension as a sector of $\C^{d}.$

\section{Notation and preliminary
results}\label{prelimi}
\subsection{Factorial and
binomial coefficients} We use
the usual multi-index
notation for factorial and
binomial coefficients. Hence,
for
$\alpha=(\alpha_1,\ldots,\alpha_d)\in\mathbb{N}^d$
we set
$\alpha!=\alpha_1!\ldots\alpha_d!$
and for
$\beta,\alpha\in\mathbb{N}^d$,
$\beta\leq\alpha$, we set
$\binom{\alpha}{\beta}=\frac{\alpha!}{\beta!(\alpha-\beta)!}$.\par
The following inequality is
standard and used often in
the sequel:
\begin{equation}\label{pr6}
\binom{\alpha}{\beta}\leq
2^{|\alpha|}.
\end{equation}
Also, we recall the identity
$$ \sum_{\stackrel{|\alpha'|=j}{\alpha'\leq\alpha}}\binom{\alpha}{\alpha'} = \binom{|\alpha|}{j}, \qquad j=0,1,\ldots, |\alpha|,$$
which follows from $\prod _{i=1}^d
(1+t)^{\alpha_i}=(1+t)^{|\alpha|}$, and
gives in particular
\begin{equation}\label{pr4}
\binom{\alpha}{\beta}\leq\binom{|\alpha|}{|\beta|},\quad
\alpha,\beta\in\mathbb{N}^d,\
\beta\leq\alpha.
\end{equation}
The last estimate implies in turn, by induction,
\begin{equation}\label{pr8}
\frac{\alpha!}{\delta_1!\ldots\delta_j!}\leq
\frac{|\alpha|!}{|\delta_1|!\ldots|\delta_j|!},\quad\alpha=\delta_1+\ldots+\delta_j,
\end{equation}
as well as
\begin{equation}\label{pr5}
\frac{\alpha!}{(\alpha-\beta)!}\leq
\frac{|\alpha|!}{|\alpha-\beta|!},\quad
\beta\leq\alpha.
\end{equation}
Finally we recall the
so-called inverse Leibniz'
formula:
\begin{equation}\label{pr3}
x^\beta\partial^\alpha
u(x)=\sum_{\gamma\leq\beta,\,\gamma\leq\alpha}\frac{(-1)^{|\gamma|}\beta!}{(\beta-\gamma)!}
\binom{\alpha}{\gamma}\partial^{\alpha-\gamma}(x^{\beta-\gamma}u(x)).
\end{equation}
\subsection{Pseudodifferential
Operators}
We collect here some basic properties of the class $\Gamma_{a}^{m}(\rd)$ defined by the
 estimates \eqref{Gsymbols} and of the corresponding operators \eqref{pr0}. Actually,
  for our purposes it is not necessary to develop a specific calculus for the analytic symbols of $\Gamma_{a}^{m}(\rd)$. We shall deduce the properties we need from those of the larger class $\Gamma^{m}(\rd)$ defined as the space of all functions $p \in C^\infty(\R^{2d})$ satisfying the following
estimates: for every
$\alpha,\beta\in\mathbb{N}^d$
there exists a constant
$C_{\alpha,\beta}>0$ such
that
\begin{equation}\label{defing}
|\partial^\beta_x\partial^\alpha_\xi
p(x,\xi)|\leq
C_{\alpha,\beta}(
1+|x|+|\xi|)^{m-|\alpha|}\langle x\rangle^{-|\beta|}
\end{equation}
for every $x,\xi\in\rd$.
Clearly $\Gamma_{a}^{m}(\rd) \subset \Gamma^{m}(\rd).$ We shall denote by
 ${\rm OP}\Gamma^{m}(\rd)$ (respectively ${\rm OP}\Gamma_{a}^{m}(\rd)$)  the class of pseudodifferential operators with symbol in $\Gamma^{m}(\rd)$ (respectively in $\Gamma_{a}^{m}(\rd)$).
 We endow $\Gamma^{m}(\rd)$ with the topology defined by the seminorms
 \[
 \|p\|^{(\Gamma^m)}_N=\sup_{|\alpha|+|\beta|\leq
 N}\sup_{(x,\xi)\in\rdd}\big\{|\partial^\alpha_\xi\partial^\beta_x
p(x,\xi)|(1+|x|+|\xi|)^{-m+|\alpha|}\langle x\rangle^{|\beta|}
\big\},\quad
N\in\mathbb{N}.
\]
The properties of $\Gamma^{m}(\rd)$ follow from the general Weyl-H{\"o}rmander's calculus in \cite[Chapter XVIII]{hormanderIII}; with the notation used there, $\Gamma^m(\rd)=S(M,g)$, with the weight $M(x,\xi)=(1+|x|^2+|\xi|^2)^{m/2}$ and the metric
 \[
 g_{x,\xi}(y,\eta)=\frac{|dy|^2}{1+|x|^2}+\frac{|d\eta|^2}{1+|x|^2+|\xi|^2}.
 \]
 We also refer the reader to \cite[Chapter 1]{nicola} for an elementary and self-contained presentation; with the notation in \cite[Definition 1.1.1]{nicola} we have $\Gamma^m(\rd)=S(M;\Phi,\Psi)$, with $M(x,\xi)$ as above and
 \[
 \Phi(x,\xi)=\langle x\rangle,\qquad \Psi(x,\xi)= (1+|x|^2+|\xi|^2)^{1/2}.
 \]
 \par
Now, if $p\in \Gamma^m(\rd)$
then $p(x,D)$ defines a
continuous map
$\cS(\rd)\to\cS(\rd)$ which
extends to a continuous map
$\cS'(\rd)\to\cS'(\rd)$. The
composition of two such
operators is therefore well
defined in $\cS(\rd)$ and in
$\cS'(\rd)$; more precisely,
if $p_1\in \Gamma^{m_1}(\rd)$
and $p_2\in \Gamma^{m_2}(\rd),$
then
$p_1(x,D)p_2(x,D)=p_3(x,D)$
with $p_3\in
\Gamma^{m_1+m_2}(\rd)$ and
the map $(p_1,p_2)\mapsto
p_3$ is continuous
$\Gamma^{m_1}(\rd)\times
\Gamma^{m_2}(\rd)\to
\Gamma^{m_3}(\rd)$. Moreover we have that
$\bigcap\limits_{m \in \R}\Gamma^{m}(\rd)= \cS(\R^{2d}).$ In
particular, operators with Schwartz symbols are globally
regularizing, i.e. they map continuously $\cS'(\rd)$ into
$\cS(\rd)$.
\par
Operators in ${\rm OP}\Gamma^m(\rd)$ are also bounded on certain weighted Sobolev spaces. We consider, for simplicity, the case of integer positive exponents (we will only need this case, see \cite{Shubin:1} for the general case). For $s\in\mathbb{N}$, we define
\begin{equation}\label{qs}
Q^s(\rd)=\{u\in L^2(\rd):\ \|u\|_{Q^s}:=\sum_{|\alpha|+|\beta|\leq s}\|x^\beta \partial^\alpha u\|_{L^2}<\infty\}.
\end{equation}
We recall that $\bigcap\limits_{s\in \N} Q^s(\rd) \cS(\rd).$ Now, if $P\in{\rm OP}\Gamma^m(\rd)$,
$m\in{\mathbb{Z}}$, $m\leq s$, we have $P: Q^s(\rd)\to
Q^{s-m}(\rd)$ continuously with
\[
\|p(x,D)\|_{\mathcal{B}(Q^s,Q^{s-m})}\leq
C\|p\|^{(\Gamma^m)}_{N}
\]
for suitable $C>0$, $N\in\mathbb{N}$ depending only on
$s,m$ and on the dimension $d$; see \cite[Proposition
1.5.5,Theorem 2.1.12]{nicola}. Moreover, for
$m\in{\mathbb{Z}}$, $m\leq s$, there exists an operator
$T\in{\rm OP}\Gamma^{-m}(\rd)$ which gives an isomorphism
$Q^{s-m}(\rd)\to Q^s(\rd)$.  \par
 We will also need the following
Schauder's estimates for the weighted Sobolev spaces
$Q^s(\rd)$ in \eqref{qs}.
\begin{proposition}\label{schauderQ}
Let $s\in\mathbb{N}$, $s>d/2$. There exists $C_s>0$ such that
\[
\|uv\|_{Q^s}\leq C_s\|u\|_{Q^s} \|v\|_{Q^s} \quad \forall u,v\in Q^s(\rd).
\]
\end{proposition}
\begin{proof}
We have
\begin{align*}
\|uv\|_{Q^s}&=\sum_{|\alpha|+|\beta|\leq s}\|x^\beta \partial^\alpha (uv)\|_{L^2}=\sum_{|\alpha|+|\beta|\leq s}\sum_{\gamma\leq\alpha}\binom{\alpha}{\gamma}\|x^\beta \partial^{\alpha-\gamma} u \cdot \partial^\gamma v\|_{L^2}\\
&\leq 2^s \sum_{|\alpha|+|\beta|\leq s}\sum_{\gamma\leq\alpha}\|x^\beta \partial^{\alpha-\gamma} u\|_{L^p}  \|\partial^\gamma v\|_{L^q},
\end{align*}
where $1\leq p,q\leq\infty$ are chosen to satisfy $\frac{1}{p}+\frac{1}{q}=\frac{1}{2}$, and $\frac{1}{2}<\frac{1}{p}+\frac{|\gamma|}{d}$, $\frac{1}{2}<\frac{1}{q}+\frac{s-|\gamma|}{d}$. This is possibile because $s>\frac{d}{2}$. Then, by the Sobolev embeddings we have
\begin{equation}\label{8ma1}
\|uv\|_{Q^s}\leq C_s\sum_{|\alpha|+|\beta|\leq s}\sum_{\gamma\leq\alpha}\|x^\beta \partial^{\alpha-\gamma} u\|_{H^{|\gamma|}}  \|\partial^\gamma v\|_{H^{s-|\gamma|}}.
\end{equation}
On the other hand,
\begin{equation}\label{8ma2}
\|x^\beta \partial^{\alpha-\gamma} u\|_{H^{|\gamma|}} \asymp\sum_{|\mu|\leq|\gamma|}\|\partial^\mu \left(x^\beta \partial^{\alpha-\gamma}u\right)\|_{L^2}\leq C'_s\|u\|_{Q^s}.
\end{equation}
Similarly,
\begin{equation}\label{8ma3}
\|\partial^\gamma v\|_{H^{s-|\gamma|}}\leq C''_s\|u\|_{Q^s}.
\end{equation}
Combining \eqref{8ma1}, \eqref{8ma2} and \eqref{8ma3} we
get the desired result.
\end{proof}

As a technical tool, we will also use the scale of weighted
Sobolev spaces
\begin{equation}\label{25-0}
H^{s_1,s_2}(\R^d)= \{u \in \mathcal{S}'(\rd):
\|u\|_{H^{s_1,s_2}}:=\|\px^{s_2}u\|_{s_1}<\infty\},
\end{equation}
defined for $s_1,s_2\in{\R}$. In particular, we need the
following result (see e.g. \cite[Definition 3.1.1 and
Theorem 3.1.5]{nicola}).
\begin{proposition}\label{25pro}
Consider a symbol $p(x,\xi)$ satisfying the estimates
\begin{equation}\label{25-1}
|\partial^\beta_x\partial^\alpha_\xi p(x,\xi)|\leq
C_{\alpha,\beta}\langle x\rangle^{n-|\beta|}\langle
\xi\rangle^{m-|\alpha|},\quad\forall\alpha,\beta\in\N^d,\
x\in\rd.
\end{equation}
Then the corresponding operator $p(x,D)$ is bounded
$H^{s_1,s_2}(\rd)\to H^{s_1-m,s_2-n}(\rd)$ for every
$s_1,s_2\in\R$, with operator norm estimated by an upper
bound of a finite number of the constants $C_{\alpha,\beta}$
appearing in \eqref{25-1}.
\end{proposition}
By using Schauder's estimates in the standard Sobolev
spaces and the inclusion $H^{s_1,s_2}(\rd)\hookrightarrow
H^{s_1}(\rd)$, valid if $s_2\geq0$, one also gets
\begin{equation}\label{25-3}
\|uv\|_{H^{s_1,s_2}}\leq C_{s_1,s_2}\|u\|_{H^{s_1,s_2}}
\|v\|_{H^{s_1,s_2}}\qquad s_1>\frac{d}{2},\quad s_2\geq0.
\end{equation}
A symbol $p\in \Gamma^{m}(\rd)$
(and the corresponding
operator) is said to be {\it
$\Gamma$-elliptic} if  it
satisfies the condition \eqref{Gammaell}.
\par
The notion of $\Gamma$-ellipticity for an operator in
${\rm OP}\Gamma^m(\rd)$ will be crucial in the subsequent
arguments because it guaranties the existence of a
parametrix $E\in{\rm OP}\Gamma^{-m}(\rd)$. Namely we have
the following result, see \cite[Theorem 1.3.6]{nicola} for the proof.
\begin{proposition} \label{para}
Let $p\in \Gamma^{m}(\rd)$ be $\Gamma$-elliptic. Then there exists an operator $E\in{\rm OP}\Gamma^{-m}(\rd)$ such that $EP=I+R$ and $PE=I+R'$, where
$R,R'$ are {\rm globally}
regularizing
pseudodifferential
operators, i.e. $R$ and $R'$
are continuous maps
$\cS'(\rd)\to \cS(\rd)$. The operator $E$ is said to be a parametrix for $P$.
\end{proposition}
\par
Finally
we point out for further
reference the following
formulae, which can be
verified by a direct
computation: for
$\alpha,\beta\in\mathbb{N}^d, u \in \cS(\rd):$
\begin{equation}\label{pr1}
x^\beta
p(x,D)u=\sum_{\gamma\leq\beta}
(-1)^{|\gamma|}\binom{\beta}{\gamma}(D^\gamma_\xi
p)(x,D)(x^{\beta-\gamma}u),
\end{equation}
\begin{equation}\label{pr2}
\partial^\alpha p(x,D)u=\sum_{\delta\leq\alpha}
\binom{\alpha}{\delta}(\partial^\delta_x
p)(x,D)\partial^{\alpha-\delta}u.
\end{equation}
\section{A space of analytic
functions}\label{secspazi} We introduce a space of
real-analytic functions in $\rd$, which extend
holomorphically on a sector in $\C^d$ and display there a
Gaussian decay.
\begin{definition}\label{classisg}
We denote by $\sg$ the space of all
functions $f\in C^\infty(\rd)$
satisfying the following condition:
there exists a constant $C>0$ such that
\begin{equation}\label{cla0}
|x^\beta\partial^\alpha f(x)|\leq
C^{|\alpha|+|\beta|+1}\mab, \quad{\rm for\ all}\
\alpha,\,\beta\in\N^d,
\end{equation}
where
\begin{equation}\label{cla0bis}
\mab=|\alpha|!^{1/2}\max\{|\alpha|,|\beta|\}!^{1/2}.
\end{equation}
\end{definition}
It is easy to verify that the space $\sg$ is closed under differentiation.

\begin{theorem}\label{estensione}
Let $f\in\sg$. Then $f$ extends to a holomorphic function
$f(x+iy)$ in the sector
\begin{equation}\label{cla3}
\mathcal{C}_\eps=\{z=x+iy\in\mathbb{C}^d:\ |y|<\eps(1+|x|)\}
\end{equation}
 of $\mathbb{C}^d$ for some $\eps>0$, satisfying there the estimates
\begin{equation}\label{cla4}
|f(x+iy)|\leq Ce^{-c|x|^2},
\end{equation}
for some constants $C>0$, $c>0$. \end{theorem}
\begin{proof}
First we show the estimates
\begin{equation}\label{cla5}
|x^\beta\partial^\alpha f(x)|\leq C^{|\alpha|+1}|\alpha|!
e^{-c|x|^2},\quad {\rm for}\ |\beta|\leq|\alpha|.
\end{equation}
Indeed, since $|x|^{2n}\leq
k^n\sum_{|\gamma|=n}|x^{2\gamma}|$ for a constant $k>0$
depending only on the dimension $d$, by \eqref{cla0} we
have (assuming $C\geq1$ in \eqref{cla0})
\begin{align*}
e^{c|x|^2} |x^\beta\partial^\alpha f(x)|&=\sum_{n=0}^\infty
\frac{(c|x|^2)^n}{n!}|x^\beta\partial^\alpha
 f(x)|\\
 &\leq \sum_{n=0}^\infty (ck)^n\sum_{|\gamma|=n}\frac{1}{|\gamma|!}|x^{\beta+2\gamma}\partial^\alpha f(x)|\\
 &\leq \sum_{n=0}^\infty (ck)^n\sum_{|\gamma|=n} C^{2|\alpha|+2|\gamma|+1}
 \frac{|\alpha|!^{1/2}(|\alpha|+2|\gamma|)!^{1/2}}{|\gamma|!}\\
 &\leq \sum_{n=0}^\infty (ck)^n\sum_{|\gamma|=n} (2C)^{2|\alpha|+2|\gamma|+1}{|\alpha|!},
 \end{align*}
 where in the last step we used the inequality
 $(|\alpha|+2|\gamma|)!\leq
 2^{|\alpha|+4|\gamma|}|\alpha|!|\gamma|!^2$, which follows
by applying twice \eqref{pr6}.
 Since the number of multi-indices $\gamma$ satisfying
  $|\gamma|=n$ does not exceed $2^{d+n-1}$, we get \eqref{cla5} for a new constant $C$, if $c$ is small enough. Now, \eqref{cla5} and the estimate $|\alpha|!\leq d^{|\alpha|}\alpha!$ give
  \begin{equation}\label{cla6}
 |\partial^\alpha f(x)|\leq C^{|\alpha|+1}\alpha!\langle x\rangle^{-|\alpha|}e^{-c|x|^2},
 \end{equation}
  for a new constant $C>0$. By considering the Taylor expansion of the function
  $f$ centered in any $x\in\rd$ and using the estimates in \eqref{cla6} we obtain the desired extension property in a sector of the type \eqref{cla3} together  with the estimates \eqref{cla4}.
  \end{proof}

In the sequel we will use the
following characterization of
the space $\sg$ in terms of
$Q^s$-based norms.\par Set,
for $f\in\cS'(\rd)$,
\begin{equation}\label{accaenne2}
S_\infty^{s,\eps}[f]=\sum_{\al,\,\beta \in \N^d}\frac{\eab}{\mab}
\|x^\be\partial^\al f\|_{Q^s},
\end{equation}
where $\mab$ is defined in \eqref{cla0bis}.
\begin{proposition}\label{stima0}
Let $f\in\sg$. Then for every
$s\in\N$ there exists $\eps>0$ such
that $S^{s,\eps}_\infty[f]<\infty$.\par
In the opposite direction, if for some $s\in\N$ there exists $\eps>0$ such that $S^{s,\eps}_{\infty}[f]<\infty$, then $f\in\sg$.
\end{proposition}
\begin{proof}
Assume $f\in\sg$. We have
\[
\|x^\be\partial^\al f\|_{Q^s}=\sum_{|\delta|+|\gamma|\leq
s}\|x^\delta\partial^\gamma\big(x^\be\partial^\al
f\big)\|_{L^2}.
\]
Now, if $M\in\mathbb{N}$ satisfies
$M>d/4$ we have
\begin{equation}\label{cla2}
\|x^\delta\partial^\gamma\big(x^\be\partial^\al
f\big)\|_{L^2}\leq C'\|(1+|x|^2)^M
x^\delta\partial^\gamma\big(x^\be\partial^\al
f\big)\|_{L^\infty}.
\end{equation}
By Leibniz' formula, \eqref{cla0} and
\eqref{pr6} we get
\[
\|x^\be\partial^\al f\|_{Q^s}\leq C_s^{|\alpha|+|\beta|+1}\mab
\]
for some constant $C_s>0$. Hence $S^{s,\eps}_\infty[f]<\infty$ if $\eps<C_s^{-1}$.\par
In the opposite direction, we may take $s=0$; hence assume $S^{0,\eps}_\infty[f]<\infty$ for some $\eps>0$. Then $\|x^\beta\partial^\alpha f(x)\|_{L^2}\leq
C^{|\alpha|+|\beta|+1}\mab$ for\ all
$\alpha,\,\beta\in\N^d$. If $M$ is an integer, $M>d/2$, we have
\[
\|x^\beta\partial^\alpha f\|_{L^\infty}\leq C\sum_{|\gamma|\leq M} \|\partial^\gamma\big(x^\beta\partial^\alpha f\big)\|_{L^2},
\]
and similarly one gets that $f\in\sg$.
\end{proof}

%-----------------------------------------------------------------------------------------------------
%-----------------------------------------------------------------------------------------------------

\section{Proof of the main result (Theorem
\ref{mainthm})}\label{secdim}
In this section we prove
Theorem \ref{mainthm}. In
fact we shall state and prove
 this result for the more general non-homogeneous equation
\begin{equation}
\label{A5.1f} Pu=f+F[u],
\end{equation}
where $P$  and $F[u]$ satisfy
the assumptions of Theorem
\ref{mainthm} and $f$ is a
function in the space $\sg$ defined in Section 3.
Moreover we can restate our
result in terms of estimates in $\sg$. Namely,
 in view of Theorem
 \ref{estensione},
it will be sufficient to
prove the following theorem.

\begin{theorem}\label{AA5.1}
Let $P=p(x,D)\in {\rm OP}\Gamma_{a}^{m}(\rd), m>0,$ be $\Gamma$-elliptic, that is \eqref{Gammaell} is
satisfied. Let $F[u]$ be of the form \eqref{A5.2} (possibly
with some factors in the product replaced by their
conjugates) and $f \in \sg.$ Assume moreover that $u\in
H^s(\rd)$, $s>d/2+\max_{k}\{|\rho_k|\}$, is a solution of
\eqref{A5.1f}. Then $u\in\sg$.
\end{theorem}
In fact we always assume that
$F[u]$ has the form in
\eqref{A5.2}, and we leave to
the reader the easy changes
when some factors of the
product in \eqref{A5.2} are
replaced by their conjugates.
\par The first step is to show that, under the assumptions of Theorem \ref{AA5.1}, the sum $S_N^{s,\ve}[u]$ is finite for every $N
\in \N$ and for some $\ve >0.$ In particular, we prove the following preliminary result.

\begin{lemma} Under the assumptions of Theorem \ref{AA5.1}, we have $u \in \mathcal{S}(\rd).$
\end{lemma}

\begin{proof}
The proof is based on a bootstrap argument in the scale of
weighted Sobolev spaces defined in \eqref{25-0}. Notice
that $H^s(\rd)=H^{s,0}(\R^d)$ and that
$\bigcap\limits_{s_1,s_2 \in
\R}H^{s_1,s_2}(\R^d)$ $=\mathcal{S}(\rd).$ To prove the lemma
it is then sufficient to show that if $u \in
H^{s_1,s_2}(\rd)$, $s_1>d/2+M,$ with $M=\max_k
\{|\rho_k|\},$ $s_2\geq0$, then $u \in H^{s_1+\tau,
s_2+\tau}$, with $\tau=\min\{m/2,1/2\}$, and to iterate this argument. We first consider
the case $m\geq 1.$ Let $E \in \textrm{OP}\Gamma^{-m}(\rd)$
be a parametrix of $P$ (Proposition \ref{para}). Applying
$E$ to both sides of the equation \eqref{A5.1f}, we get
$$u=-Ru+Ef+EF[u],$$
where $R$ is globally regularizing. In particular, we have
$Ru\in \mathcal{S}(\rd)$ and $Ef \in \mathcal{S}(\rd)$
because $f \in \mathcal{S}(\rd).$ Concerning the nonlinear
term, we observe that, since $-m\leq -M-h-1,$ then the
symbol $e(x,\xi)$ of $E$ satisfies the following estimates
\begin{eqnarray*}
|\partial_{\xi}^{\alpha}\partial_x^{\beta}e(x,\xi)|&\leq& C_{\alpha \beta}
(1+|x|+|\xi|)^{-M-h-1-|\alpha|}\px^{-|\beta|} \\
&\leq& C_{\alpha
\beta}\pxi^{-M-1/2-|\alpha|}\px^{-h-1/2-|\beta|}.
\end{eqnarray*} It follows from Proposition \ref{25pro} that
$E:H^{s_1-M,s_2-h}(\rd)\rightarrow
H^{s_1+1/2,s_2+1/2}(\rd)$
 continuously for every $s_1,s_2 \in \R$. In particular, in view of \eqref{estnonlincoeff}, we have
\begin{eqnarray*} \| EF[u]\|_{H^{s_1+1/2,s_2+1/2}} &=&
\| E\sum_{h,l,\rho_{1},\ldots,
\rho_{l}}F_{h,l,\rho_{1}\ldots \rho_{l}}(x)
 \prod_{k=1}^l\partial^{\rho_{k}}u\|_{H^{s_1+1/2,s_2+1/2}}
 \\ &\leq&C \sum_{h,l,\rho_{1},\ldots, \rho_{l}}\|F_{h,l,\rho_{1}\ldots \rho_{l}}(x)
 \prod_{k=1}^l\partial^{\rho_{k}}u\|_{H^{s_1-M,s_2-h}} \\
 &\leq& C' \|\prod_{k=1}^{l} \partial^{\rho_k}u\|_{H^{s_1-M,s_2}} \leq C'' \|u\|_{H^{s_1,s_2}}^{l}<\infty.\end{eqnarray*}
by Schauder's estimates \eqref{25-3}, because $s_1-M >d/2.$
The case $0<m<1$ is completely similar, considering that in
this case $h=M=0$ and the symbol $e(x,\xi)$ of $E$
satisfies the estimates
$$|\partial_{\xi}^{\alpha}\partial_x^{\beta}e(x,\xi)|\leq
C_{\alpha \beta} \pxi^{-m/2-|\alpha|}\px^{-m/2-|\beta|},$$
so that $E$ maps continuously $H^{s_1,s_2}(\rd)$ into
$H^{s_1+m/2,s_2+m/2}(\rd).$
\end{proof}
\par
In order to prove Theorem \ref{AA5.1} it suffices to verify
that $S^{s,\eps}_\infty[u]<\infty$ for some $s \geq 0,
\ve>0$, in view of Proposition \ref{stima0}. This will be
achieved by an iteration argument involving the partial
sums of the series in \eqref{accaenne2}, that are
\begin{equation}\label{accaenne}
S_N^{s,\eps}[f]=\sum_{|\al|+|\beta|\leq N}\frac{\eab}{\mab}
\|x^\be\partial^\al f\|_{Q^s},
\end{equation}
where $\mab$ is defined in \eqref{cla0bis}.

\subsection{Proof of Theorem \ref{AA5.1}}
We need
several estimates to which we
address now.

\begin{proposition}\label{stima1}
Let $R\in{\rm OP}\Gamma^{-1}(\rd)$. Then for every $s\in\R$
there exists a constant $C_s>0$ such that, for every
$\eps>0, N \geq 1$ and $u \in \cS(\rd)$, we have
\[
\sum_{0<|\alpha|+|\beta|\leq N}\frac{\eab}{\mab}\|R(x^\beta\partial^\alpha
u)\|_{Q^s}\leq C_s\eps
S^{s,\eps}_{N-1}[u].
\]
\begin{proof}
We first estimate the terms with $\beta=0$, hence
$\alpha\not=0$. Let $k\in\{1,\ldots,d\}$ such that
$\alpha_{k}\not=0$. Since $R\circ \partial_{k}\in{\rm
OP}\Gamma^{0}(\rd)$ is bounded on $Q^s(\rd)$ we
have\footnote{We denote by $e_k$ the $k$th vector of the
standard basis of $\rd$.}
\[
\frac{\eps^{|\alpha|}}{|\alpha|!} \|R(\partial^\alpha
u)\|_{Q^s}\leq
C_s\eps\frac{\eps^{|\alpha|-1}}{|\alpha|!}\|\partial^{\alpha-e_{k}}u\|_{Q^s}.
\]

On the other hand, when $\beta\not=0$, hence
$\beta_j\not=0$ for some $j\in\{1,\ldots,d\}$, we use the
fact that $R\circ x_{j}\in{\rm OP}\Gamma^{0}(\rd)$ is
bounded on $Q^s(\rd)$. We get
\[
\frac{\eps^{|\alpha|+|\beta|}}{M(\alpha,\beta)} \|R(x^\beta
\partial^\alpha u)\|_{Q^s}\leq
C_s\eps\frac{\eps^{|\alpha|+|\beta|-1}}{M(\alpha,\beta)}\|x^{\beta-e_{j}}\partial^\alpha
u\|_{Q^s}.
\]
Since $M(\alpha,\beta)\geq M(\alpha,\beta-e_j)$, this gives
the desired result.
\end{proof}
\end{proposition}
\begin{proposition}\label{commutatore}
Let $P=p(x,D)$ be a pseudodifferential with symbol
$p(x,\xi)$ satisfying the estimates \eqref{Gsymbols}, with
$m\geq0$. Let $E\in{\rm OP}\Gamma^{-m}(\rd)$. Then for
every $s\in\R$ there exists a constant $C_s>0$ such that,
for every $\epsilon$ small enough, $N \geq 1$ and $u\in
\cS(\rd)$, we have
\begin{equation}\label{nn5}
\sum_{0<|\alpha|+|\beta|\leq N}\frac{\eab}{\mab}\|E[P,x^\beta
\partial^\alpha]u\|_{Q^s}\leq C_s\eps
S^{s,\eps}_{N-1}[u].
\end{equation}
\end{proposition}
\begin{proof}

We estimate separately each term arising in the sum
\eqref{nn5}. We write
\[
[P,x^\beta
\partial^\alpha]=[P,x^\beta]\partial^\alpha+x^\beta[P,\partial^\alpha].
\]
Hence, using \eqref{pr1}, \eqref{pr2}, we get
\begin{multline}\label{c0}
[P,x^\beta
\partial^\alpha]u=\sum_{0\not=\gamma_0\leq\beta}
(-1)^{|\gamma_0|+1}\binom{\beta}{\gamma_0}(D^{\gamma_0}_\xi
p)(x,D)(x^{\beta-\ga_0}\partial^\alpha
u)\\
- \sum_{0\not=\delta\leq\alpha} \binom{\alpha}{\delta}
x^\beta\partial^\delta_x p(x,D)
\partial^{\alpha-\delta} u.
\end{multline}
Given $\beta$, $\delta$, let $\tilde{\delta}$ be a
multi-index of maximal length among those satisfying
$|\tilde{\delta}|\leq |\delta|$ and
$\tilde{\delta}\leq\beta$ (hence, if
$|\tilde{\delta}|<|\delta|$ then $\beta-\tilde{\delta}=0$).
Writing $x^\beta =x^{\tilde{\delta}}x^{\beta-\tilde{\delta}}$ in the last
term of \eqref{c0} and using again \eqref{pr1} we get
\begin{equation}\label{c0bis}
[P,x^\beta
\partial^\alpha]u=\sum_{\delta\leq\alpha}\sum_{\stackrel{\gamma_0\leq\beta-\tilde{\delta}}{
(\delta,\ga_0)\not=(0,0)}}(-1)^{|\gamma_0|+1}\binom{\beta-\tilde{\delta}}
{\gamma_0} \binom{\alpha}{\delta} x^{\tid}(D^{\ga_0}_\xi
\partial^\delta_x p)(x,D)(
x^{\beta-\tid-\ga_0}\partial^{\alpha-\delta} u).
\end{equation}

We now work out this formula to obtain a useful representation of the commutator $[P,x^\beta
\partial^\alpha]$. Namely, we look at the operator $
x^{\beta-\tid-\ga_0}\partial^{\alpha-\delta}$. Given
$\ga_0$, $\alpha$, $\delta$, let $\tg_0$ be a multi-index, to be chosen later on,
satisfying $|\tg_0|\leq|\ga_0|$ and $\tg_0\leq\alpha-\delta$. We
write, by the inverse Leibniz formula \eqref{pr3},
\begin{multline}
x^{\beta-\tilde{\delta}-\ga_0}\partial^{\alpha-\delta}
=x^{\beta-\tilde{\delta}-\ga_0}\partial^{\tg_0}\partial^{\alpha-\delta-\tg_0}
=\partial^{\tg_0}\circ
x^{\beta-\tilde{\delta}-\ga_0}\partial^{\alpha-\delta-\tg_0}\\ +\sum_{\stackrel{0\not=\ga_1
\leq\beta-\tilde{\delta}-\ga_0}{
\ga_1\leq\tilde{\ga}_0}}\frac{(-1)^{|\ga_1|}(\beta-\tid-\ga_0)!}{(\beta-\tid-\ga_0-\ga_1)!}
\binom{\tg_0}{\ga_1}\partial^{\tg_0-\ga_1}\circ
x^{\beta-\tilde{\delta}-\ga_0-\ga_1}\partial^{\alpha-\delta-\tg_0}.
\end{multline}
We now look at the operator $
x^{\beta-\tid-\ga_0-\ga_1}\partial^{\alpha-\delta-\tg_0}$.
We denote by $\tg_1$ a multi-index, to be chosen later on,
satisfying $|\tg_1|\leq|\ga_1|$, $\tg_1\leq
\alpha-\delta-\tg_0$. Applying again the inverse Leibniz
formula we have
\begin{multline}
x^{\beta-\tilde{\delta}-\ga_0-\ga_1}\partial^{\alpha-\delta-\tg_0}
=x^{\beta-\tilde{\delta}-\ga_0-\ga_1}\partial^{\tg_1}
\partial^{\alpha-\delta-\tg_0-\tg_1}\\
=\partial^{\tg_1}\circ
x^{\beta-\tilde{\delta}-\ga_0-\ga_1}\partial^{\alpha-\delta-\tg_0-\tg_1}
\\
+\sum_{\stackrel{0\not=\ga_2
\leq\beta-\tilde{\delta}-\ga_0-\ga_1}{
\ga_2\leq\tg_1}}\frac{(-1)^{|\ga_2|}(\beta-\tid-\ga_0-\ga_1)!}
{(\beta-\tid-\ga_0-\ga_1-\ga_2)!}
\binom{\tg_1}{\ga_2}\partial^{\tg_1-\ga_2}\circ
x^{\beta-\tilde{\delta}-\ga_0-\ga_1-\ga_2}\partial^{\alpha-\delta-\tg_0-\tg_1}.
\end{multline}
Continuing in this way and
substituting all in
\eqref{c0bis} we get
\begin{eqnarray}\label{c0ter}
[P,x^\beta
\partial^\alpha]u&=&\sum_{\delta\leq\alpha}\sum_{j=0}^r
\sum_{\stackrel{\gamma_0\leq\beta-\tilde{\delta}}{
(\delta,\ga_0)\not=(0,0)}}
\sum_{\stackrel{0\not=\ga_1\leq\beta-\tid-\ga_0}{
\ga_1\leq\tg_0}} \hskip-7pt \cdots \hskip-7pt
\sum_{\stackrel{0\not=\ga_j\leq\beta-\tid-\ga_0-\ldots-\ga_{j-1}}{
\ga_j\leq \tg_{j-1}}} \hskip-12pt
C_{\alpha,\beta,\delta,\ga_0,\ga_1,\ldots,\ga_j}\\
&&\times \nonumber
p_{\alpha,\beta,\delta,\ga_0,\ga_1,\ldots,\ga_j}(x,D)\big(
x^{\beta-\tid-\ga_0-\ldots-\ga_j}\partial^{\alpha-\delta-\tg_0-\ldots-\tg_j}
u\big),
\end{eqnarray}
where $\tg_j$ satisfy $|\tg_j|\leq |\ga_j|$
and $\tg_j\leq
\alpha-\delta-\tg_0-\ldots-\tg_{j-1}$,
\begin{equation}
p_{\alpha,\beta,\delta,\ga_0,\ga_1,\ldots,\ga_j}(x,\xi)=x^{\tid}
\big(D^{\ga_0}_\xi\partial^\delta_x
p\big)(x,\xi)\xi^{\tg_0-\ga_1+\tg_1-\ldots-\ga_j+\tg_j},\quad
j\geq 0, \label{ordinecomm}
\end{equation} and
\begin{align}\label{c1bis}
|C_{\alpha,\beta,\delta,\ga_0,\ga_1,\ldots,\ga_j}|&=\frac{\alpha!(\beta-\tid)!}
{(\alpha-\delta)!\delta!\ga_0!(\beta-\tid-\ga_0-\ldots-\ga_j)!}\prod_{k=1}^j\binom{\tg_{k-1}}{\ga_k}
\nonumber\\
&\leq
\frac{|\alpha|!|\beta-\tid|!}
{|\alpha-\delta|!\delta!\ga_0!|\beta-\tid-\ga_0-\ldots-\ga_j|!}
2^{|\tg_0+\ldots+\tg_{j-1}|},
\end{align}
cf. \eqref{pr5} and \eqref{pr6}.
(If $j=0$ in \eqref{c1bis} we
mean that there are not the
binomial factors, nor the
power of $2$).
 Observe that,
since we have $\ga_j\not=0$
for every $j\geq1$, this
procedure in fact stops after
a finite number $r$ of steps.\par
We now study separately the terms with $|\alpha|\geq|\beta|$ and $|\beta|>|\alpha|$.

\par\medskip\noindent {\bf Case
$|\alpha|\geq|\beta|$.}\par We use the formula
\eqref{c0ter}, where now we choose $\gamma_0$ satisfying,
in addition, $|\tg_0|=|\gamma_0|$. Such a multi-index
exists, because $|\alpha|\geq|\beta|$. Similarly, at each subsequent step we can choose
$\tg_j$, $j\geq1$, satisfying in addition
$|\tg_j|=|\gamma_j|$. \par Now we observe that, by
\eqref{Gsymbols}, \eqref{pr6}, and Leibniz' formula, for
every $\theta,\sigma\in\mathbb{N}^d$ we have
\begin{equation}\label{c2}
|\partial^\theta_\xi\partial^\sigma_x
p_{\alpha,\beta,\delta,\ga_0,\ga_1,\ldots,\ga_j}(x,\xi)|\leq
C^{|\ga_0|+|\delta|+1}\ga_0!\delta!(1+|x|+|\xi|)^{m-|\theta|}\langle x\rangle^{-|\sigma|},
\end{equation}
for some constant $C$ depending only on $\theta$ and
$\sigma$. In fact, $|\tid|\leq |\delta|$,
$|\tg_0-\ga_1+\tg_1-\ldots-\ga_j+\tg_j|=|\tg_0|=|\ga_0|$,
and the powers of $|\delta|$ and $|\gamma_0|$ which arise
can be estimated by $C^{|\ga_0|+|\delta|+1}$ for some
$C>0$.\par
 We now use these last bounds to estimate
 $E\circ p_{\alpha,\beta,\delta,\ga_0,\ga_1,\ldots,\ga_j}(x,D)$. To this end,
 observe that this operator
 belongs to ${\rm OP}\Gamma^{0}(\rd)$, and
 therefore its norm as a
 bounded operator on
 $Q^s(\rd)$ is estimated by a
 seminorm of its symbol in
 $\Gamma^{0}(\rd)$, depending only on
 $s$ and $d$. Such a seminorm
 is in turn estimated by the
 product of a seminorm of the
 symbol of $E$ in
 $\Gamma^{-m}(\rd)$ and a
 seminorm of
 $p_{\alpha,\beta,\delta,\ga_0,\ga_1,\ldots,\ga_j}$
 in $\Gamma^{m}(\rd)$, again
 depending only on $s,d$.
 Hence, from \eqref{c2}
 we get
 \begin{equation}\label{c3}
 \|E\circ p_{\alpha,\beta,\delta,\ga_0,\ga_1,\ldots,\ga_j}(x,D)
 \|_{\mathcal{B}(Q^s)}\leq
 C_s^{|\ga_0|+|\delta|+1}\ga_0!\delta!.
\end{equation}
Since $|\tg_k|=|\ga_k|$,
$0\leq k\leq j$, we have
\begin{equation}\label{c5}
\frac{|\beta-\tilde{\delta}|!|\alpha-\delta-\tg_0-\ldots-\tg_j|!}{|\alpha-
\delta|!|\beta-\tid-\ga_0-\ldots-\ga_j|!
}\leq1,
\end{equation}
(recall that if
$|\tid|<|\delta|$ then
$\beta-\tid=\ga_0=\ldots=\ga_j=\tg_0=\ldots=\tg_j=0$).\par
By \eqref{c1bis}, \eqref{c3}
\eqref{c5}, we get in this
case
\begin{multline}\label{c4bis}
\frac{\eps^{|\alpha|+|\beta|}}{|\alpha|!}
|C_{\alpha,\beta,\delta,\ga_0,\ga_1,\ldots,\ga_j}|
\|E\circ p_{\alpha,\beta,\delta,\ga_0,\ga_1,\ldots,\ga_j}(x,D)(
x^{\beta-\tid-\ga_0-\ldots-\ga_j}\partial^{\alpha-\delta-\tg_0-\ldots-\tg_j}
u)\|_{Q^s}\\
\leq
C_s(C_s\epsilon)^{|\delta|+|\tilde{\delta}|+|\ga_0+\ldots+\ga_j|+|\tg_0+\ldots+\tg_{j}|}
\frac{\eps^{|\alpha|+|\beta|-|\delta|-|\tilde{\delta}|
-|\ga_0+\ldots+\ga_j|-|\tg_0+\ldots+\tg_{j}|}}
{|\alpha-\delta-\tg_0-\ldots-\tg_j|!}\\
\times
\|x^{\beta-\tid-\ga_0-\ldots-\ga_j}\partial^{\alpha-\delta-\tg_0-\ldots-\tg_j}
u\|_{Q^s}.
\end{multline}
Now, the assumption $|\alpha|\geq|\beta|$ and the choice of $\tilde{\delta}$ and $\tg_j$, $j\geq0$, imply $M(\alpha,\beta)=|\alpha|!$ and $|\beta-\tid-\ga_0-\ldots-\ga_j| \leq
|\alpha-\delta-\tg_0-\ldots-\tg_j|$.
Hence
\[
M(\alpha-\delta-\tg_0-\ldots-\tg_j,\beta-\tid-\ga_0-\ldots-\ga_j)=|\alpha-\delta-\tg_0-\ldots-\tg_j|!.
\]
We can therefore rewrite \eqref{c4bis} as
\begin{multline}\label{c5bis}
\frac{\eps^{|\alpha|+|\beta|}}{M(\alpha,\beta)}
|C_{\alpha,\beta,\delta,\ga_0,\ga_1,\ldots,\ga_j}|
\|E\circ p_{\alpha,\beta,\delta,\ga_0,\ga_1,\ldots,\ga_j}(x,D)(
x^{\beta-\tid-\ga_0-\ldots-\ga_j}\partial^{\alpha-\delta-\tg_0-\ldots-\tg_j}
u)\|_{Q^s}  \\ \leq
C_s(C_s\epsilon)^{|\delta|+|\tilde{\delta}|+|\ga_0+\ldots+\ga_j|+|\tg_0+\ldots+\tg_{j}|} \times
\\ \times \frac{\eps^{|\alpha|+|\beta|-|\delta|-|\tilde{\delta}|
-|\ga_0+\ldots+\ga_j|-|\tg_0+\ldots+\tg_{j}|}}
{M(\alpha-\delta-\tg_0-\ldots-\tg_j,\beta-\tid-\ga_0-\ldots-\ga_j)}
\|x^{\beta-\tid-\ga_0-\ldots-\ga_j}\partial^{\alpha-\delta-\tg_0-\ldots-\tg_j}
u\|_{Q^s}.
\end{multline}

We now perform the change of variables
$\tilde{\alpha}=\alpha-\delta-\tg_0-\ldots-\tg_j$,
$\tilde{\beta}=\beta-\tid-\ga_0-\ldots-\ga_j$. In fact, the
map
$(\alpha,\beta,\delta,j,\gamma_0,\gamma_1,\ldots,\gamma_j)\to
(\tilde{\alpha},\tilde{\beta},\delta,j,\gamma_0,\gamma_1,\ldots,\gamma_j)$
defined in this way is not injective, because of the
presence of $\tilde{\delta},\tg_0,\ldots,\tg_j$ (of course,
one should think of $\tilde{\delta}$ as a function
of $\alpha,\beta,\delta$, and to every $\tg_j$, $j\geq0$,
as a function of $\alpha,\beta,\delta,\gamma_{k}$, $k\leq
j$). Anyhow, since $|\tilde{\delta}|\leq|\delta|$ and
$|\tg_j|=|\ga_j|$, the number of pre-images  of a given
point is at most
$2^{|\delta|+|\ga_0|+\ldots+|\ga_j|+d(j+2)}$. Hence we
deduce from  \eqref{c0ter} and \eqref{c5bis} that, if
$\epsilon$ is small enough,
\begin{multline}\label{ultima0}
\sum_{\stackrel{|\alpha|+|\beta|\leq N}{
|\alpha|\geq|\beta|}}\frac{\eab}{\mab}
\|E[P,x^\beta\partial^\alpha]u\|_{Q^s} \leq \\
 2^dC_s \sum_{|\tilde{\alpha}|+|\tilde{\beta}|\leq N-1}
\frac{\eps^{|\tilde{\alpha}|+|\tilde{\beta}|}}{M(\tilde{\alpha},\tilde{\beta})}\|x^{\tilde{\beta}}\partial^{\tilde{\alpha}}
u\|_{Q^s} \sum_{j=0}^r
2^{d(j+1)}\sum_{\delta}\sum_{\stackrel{\ga_1\not=0,
\ldots,\ga_j\not=0}{ \ga_0:\,(\delta,\ga_0)\not=(0,0)}}
\hskip-17pt
(2C_s\epsilon)^{|\delta|+|\ga_0+\ga_1+\ldots+\ga_j|}\\
\leq S^{s,\eps}_{N-1}[u]\sum_{j=0}^r (C'_s\eps)^{j+1}
 \leq C''_s\eps
 S^{s,\eps}_{N-1}[u].
\end{multline}

%%%%%%%%%%%%%%%%%%%%%%%%%%%%%%

\par\medskip\noindent
 {\bf
Case $|\beta|>|\alpha|$}. Here it is convenient to get
separate estimates when $|x|$ is large or small at the
scale $|\beta|^{1/2}$. To make this precise, consider a
function $\varphi\in C^\infty_0 (\rd)$, $\varphi(x)=1$ for $|x|\leq
1$ and $\varphi(x)=0$ for $|x|\geq 2$. Let then
\[
\varphi_\beta(x)=\varphi\Big(\frac{x}{|\beta|^{1/2}}\Big),
\]
(notice that $\beta\not=0$, because of our hypothesis $|\beta|>|\alpha|$).
Hence $\varphi_\beta(x)=1$ for $|x|\leq|\beta|^{1/2}$ and
$\varphi_\beta(x)=0$ for $|x|\geq 2|\beta|^{1/2}$. Moreover, we have
\begin{equation}\label{cutoff1}
|\partial^\gamma \varphi_\beta(x)|\leq
C_\gamma|\beta|^{-|\gamma|/2},\qquad \gamma\in\mathbb{N}^{d},\
x\in\rd.
\end{equation}
for constants $C_\gamma>0$.\par
We write
\[
[P,x^\beta\partial^\alpha]= \varphi_\beta(x)[P,x^\beta\partial^\alpha]+ (1-\varphi_\beta(x))[P,x^\beta\partial^\alpha],
\]
and we split consequently the terms in \eqref{nn5}.\par\medskip

 {\bf Estimate of $\displaystyle \frac{\eab}{\mab}\|E\big((1-\varphi_\beta(x))[P,x^\beta
\partial^\alpha]u\big)\|_{Q^s}
$}.

We use the expression in \eqref{c0bis} for $[P,x^\beta\partial^\alpha]$, and we split the second sum, by considering separately the terms with $|\beta-\tilde{\delta}-\gamma_0|\leq|\alpha-\delta|$ or $|\beta-\tilde{\delta}-\gamma_0|>|\alpha-\delta|$. This is equivalent to saying $|\beta-\gamma_0|\leq|\alpha|$ or $|\beta-\gamma_0|>|\alpha|$ because $|\beta|>|\alpha|$ implies $|\tilde{\delta}|=|\delta|$. Moreover we apply the iterative argument at the beginning of the present proof to the terms with $|\beta-\gamma_0|\leq|\alpha|$. We obtain

\begin{equation}\label{c00bis}
 \frac{\eab}{\mab}\|E\big((1-\varphi_\beta(x))[P,x^\beta
\partial^\alpha]u\big)\|_{Q^s}\leq (I)+(II)
\end{equation}

where
\begin{multline}\label{c000bis}
(I)=\sum_{\delta\leq\alpha}\sum_{\stackrel{\gamma_0\leq\beta-\tilde{\delta}:}{
(\delta,\ga_0)\not=(0,0),\ |\beta-\gamma_0|>|\alpha|}} \frac{\eab}{\mab}\binom{\beta-\tilde{\delta}}
{\gamma_0} \binom{\alpha}{\delta}\\
\times \|E\big(x^{\tid}(1-\varphi_\beta(x))(D^{\ga_0}_\xi
\partial^\delta_x p)(x,D)(
x^{\beta-\tid-\ga_0}\partial^{\alpha-\delta} u)\big)\|_{Q^s},
\end{multline}
whereas
\begin{multline}\label{c0000bis}
(II)=\sum_{\delta\leq\alpha}\sum_{j=0}^r
\sum_{\stackrel{\gamma_0\leq\beta-\tilde{\delta}:}{
(\delta,\ga_0)\not=(0,0),\ |\beta-\gamma_0|\leq|\alpha|}}
\sum_{\stackrel{0\not=\ga_1\leq\beta-\tid-\ga_0}{
\ga_1\leq\tg_0}}\cdots
\sum_{\stackrel{0\not=\ga_j\leq\beta-\tid-\ga_0-\ldots-\ga_{j-1}}{
\ga_j\leq \tg_{j-1}}} \frac{\eab}{\mab} \times \\
|C_{\alpha,\beta,\delta,\ga_0,\ga_1,\ldots,\ga_j}|\|E\big((1-\varphi_\beta(x))p_{\alpha,\beta,\delta,\ga_0,\ga_1,\ldots,\ga_j}(x,D)\big(
x^{\beta-\tid-\ga_0-\ldots-\ga_j}\partial^{\alpha-\delta-\tg_0-\ldots-\tg_j}
u\big)\big)\|_{Q^s},
\end{multline}
where the multi-indices $\tg_j$ will be chosen later on,
satisfying $|\tg_j|\leq |\ga_j|$ and $\tg_j\leq
\alpha-\delta-\tg_0-\ldots-\tg_{j-1}$; the constants
$C_{\alpha,\beta,\delta,\ga_0,\ga_1,\ldots,\ga_j}$ satisfy
\eqref{c1bis}, and \eqref{ordinecomm} holds.

\par\medskip\noindent
{\bf Estimate of the terms in {\it (I)} (hence $ |\beta-\gamma_0|>|\alpha|$)}

Since $|\tilde{\delta}|=|\delta|$, by Leibniz' formula, \eqref{Gsymbols} and
\eqref{pr6},
for
every
$\theta,\sigma\in\mathbb{N}^d$
we have
\begin{equation}\label{c2bis0}
|\partial^\theta_\xi\partial^\sigma_x(
x^{\tilde{\delta}}(D^{\gamma_0}_\xi \partial^\delta_x p)(x,\xi))|\leq
C^{|\ga_0|+|\delta|+1}\ga_0!\delta!(1+|x|+|\xi|)^{m-|\gamma_0|-|\theta|}\langle x\rangle^{-|\sigma|}
\end{equation}
for some constant $C$
depending only on $\theta$
and $\sigma$.

Since $1-\varphi_\beta(x)$ is supported where $|x|\geq |\beta|^{1/2}$, using Leibniz' formula again and \eqref{cutoff1} we get

\begin{multline*}\left|\partial^\theta_\xi\partial^\sigma_x(
x^{\tilde{\delta}}(1-\varphi_\beta(x))(D^{\gamma_0}_\xi \partial^\delta_x p)(x,\xi))\right|\leq \\
C^{|\ga_0|+|\delta|+1}\ga_0!\delta!|\beta|^{-\frac{|\gamma_0|}{2}}(1+|x|+|\xi|)^{m-|\theta|}\langle x\rangle^{-|\sigma|}\end{multline*}
for some constant $C$
depending only on $\theta$
and $\sigma$.
As a consequence,
\begin{equation}\label{c3bis2}
 \|E\circ (x^{\tilde{\delta}}(1-\varphi_\beta(x))(D^{\gamma_0}_\xi \partial^\delta_x p)(x,D))
 \|_{\mathcal{B}(Q^s)}\leq
 C_s^{|\ga_0|+|\delta|+1}\ga_0!\delta! |\beta|^{-\frac{|\gamma_0|}{2}}.
\end{equation}
On the other hand, we have
\begin{equation}\label{c3bis20}
\binom{\beta-\tilde{\delta}}{\gamma_0}\binom{\alpha}{\delta}\leq\frac{|\beta-\tilde{\delta}|!|\alpha|!}{|\beta-\tilde{\delta}-\gamma_0|!|\alpha-\delta|!\gamma_0!\delta!}
\end{equation}
as well as

\begin{multline}\label{c3bis21}
\frac{1}{|\beta|!^{1/2}|\alpha|!^{1/2}}\frac{|\beta-\tilde{\delta}|!|\alpha|!}{|\beta-\tilde{\delta}-\gamma_0|!|\alpha-\delta|!}|\beta-\tilde{\delta}-\gamma_0|!^{1/2}|\alpha-\delta|!^{1/2}|\beta|^{-|\gamma_0|/2}\\
=\underbrace{\Big(\frac{|\beta-\tilde{\delta}|!|\alpha|!}{|\alpha-\delta|!|\beta|!}\Big)^{1/2}}_{\leq 1}\underbrace{\Big(\frac{|\beta-\tilde{\delta}|!}{|\beta-\tilde{\delta}-\gamma_0|!}|\beta|^{-|\gamma_0|}\Big)^{1/2}}_{\leq1}\leq 1.
\end{multline}
By \eqref{c3bis2}, \eqref{c3bis20} and \eqref{c3bis21} we
obtain

\begin{multline}\label{c6bisbis}
\frac{\eps^{|\alpha|+|\beta|}}{|\alpha|!^{1/2}|\beta|!^{1/2}}
\binom{\beta-\tilde{\delta}}{\gamma_0}\binom{\alpha}{\delta} \|E\Big(x^{\tilde{\delta}}(1-\varphi_\beta(x))(D^{\gamma_0}_\xi \partial^\delta_x p)(x,D) u\Big)
 \|_{Q^s}\\
\leq C_s(C_s\epsilon)^{2|\delta|+|\ga_0|}
\frac{\eps^{|\alpha|+|\beta|-2|\delta| -|\ga_0|}}
{|\alpha-\delta|!^{1/2}|\beta-\tilde{\delta}-\ga_0|!^{1/2}}
\|x^{\beta-\tid-\ga_0}\partial^{\alpha-\delta} u\|_{Q^s}.
\end{multline}
Since $|\beta|>|\alpha|$ and $|\beta-\gamma_0|>|\alpha|$,
then we have $M(\alpha,\beta)=|\alpha|!^{1/2}|\beta|!^{1/2}$ and
$M(\alpha-\delta,\beta-\tilde{\delta}-\gamma_0)|\alpha-\delta|!^{1/2}|\beta-\tilde{\delta}-\ga_0|!^{1/2}$,
so that \eqref{c6bisbis} can be rephrased as

\begin{multline}\label{c6bisbis2}
\frac{\eps^{|\alpha|+|\beta|}}{M(\alpha,\beta)}
\binom{\beta-\tilde{\delta}}{\gamma_0}\binom{\alpha}{\delta} \|E(x^{\tilde{\delta}}(1-\varphi_\beta(x))(D^{\gamma_0}_\xi \partial^\delta_x p)(x,D) u)
 \|_{Q^s}\\
\leq C_s(C_s\epsilon)^{2|\delta|+|\ga_0|}
\frac{\eps^{|\alpha|+|\beta|-2|\delta| -|\ga_0|}}
{M(\alpha-\delta,\beta-\tilde{\delta}-\ga_0)}
\|x^{\beta-\tid-\ga_0}\partial^{\alpha-\delta} u\|_{Q^s}.
\end{multline}

\par\medskip\noindent
{\bf Estimate of the terms in {\it (II)} (hence
$|\beta-\gamma_0|\leq|\alpha|$)}.\par\medskip\noindent In
the iterative argument which led to \eqref{c0000bis}, we
choose the multi-indices $\tg_j$, $j\geq0$, in the
following way:
 $\tg_0$ is a multi-index satisfying,
  in addition, $|\beta-\tilde{\delta}-\gamma_0|=|\alpha-\delta-\tg_0|$.
  Such a multi-index exists, because $|\tilde{\delta}|=|\delta|$, $|\beta|>|\alpha|$ and $|\beta-\gamma_0|\leq|\alpha|$; moreover $|\gamma_0|-|\tg_0|=|\beta|-|\alpha|$. Similarly, we can choose $\tg_1$ satisfying $|\beta-\tilde{\delta}-\gamma_0-\gamma_1|=|\alpha-\delta-\tg_0-\tg_1|$; in particular $|\tg_1|=|\gamma_1|$.  In general we can choose $\tg_j$ such that
\begin{equation}\label{bede}
|\beta-\tilde{\delta}-\gamma_0-\ldots-\gamma_j|=|\alpha-\delta-\tg_0-\ldots-\tg_j|;
\end{equation}
hence $|\tg_j|=|\gamma_j|$ if $j\geq1$. \par

Notice that now in \eqref{ordinecomm} we have $|\tilde{\delta}|=|\delta|$
and $|\tg_0-\ga_1+\tg_1-\ldots-\ga_j+\tg_j|=|\tg_0|$. Hence, since $|\gamma_0|-|\tg_0|=|\beta|-|\alpha|$,
by \eqref{Gsymbols},
\eqref{pr6},
and Leibniz' formula, for
every
$\theta,\sigma\in\mathbb{N}^d$
we have
\begin{equation}\label{c2bis2}
|\partial^\theta_\xi\partial^\sigma_x
p_{\alpha,\beta,\delta,\ga_0,\ga_1,\ldots,\ga_j}(x,\xi)|\leq
C^{|\ga_0|+|\delta|+1}\ga_0!\delta!(1+|x|+|\xi|)^{m-|\beta|+|\alpha|-|\theta|}\langle x\rangle^{-|\sigma|},
\end{equation}
for some constant $C$
depending only on $\theta$
and $\sigma$.

Since $1-\varphi_\beta(x)$ is supported where $|x|\geq|\beta|^{1/2}$, using Leibniz' formula again and \eqref{cutoff1} we get

\begin{eqnarray*}|\partial^\theta_\xi\partial^\sigma_x
((1-\varphi_\beta(x))p_{\alpha,\beta,\delta,\ga_0,\ga_1,\ldots,\ga_j}(x,\xi))| &\leq &
C^{|\ga_0|+|\delta|+1}|\beta|^{-\frac{|\beta|-|\alpha|}{2}}\ga_0!\delta! \\ &\times& (1+|x|+|\xi|)^{m-|\theta|}\langle x\rangle^{-|\sigma|} \end{eqnarray*}
for some new constant $C$
depending only on $\theta$
and $\sigma$.
We obtain

\begin{equation}\label{c3bis}
 \|E\circ ((1-\varphi_\beta(x))p_{\alpha,\beta,\delta,\ga_0,\ga_1,\ldots,\ga_j}(x,D))
 \|_{\mathcal{B}(Q^s)}\leq
 C_s^{|\ga_0|+|\delta|+1}\ga_0!\delta! |\beta|^{-\frac{|\beta|-|\alpha|}{2}}.
\end{equation}

Moreover we have
\begin{equation}\label{c4bis2}
\frac{1}{|\alpha|!^{1/2}|\beta|!^{1/2}}\frac{|\alpha|!|\beta-\tilde{\delta}|!}{|\alpha-\delta|!}|\beta|^{-\frac{|\beta|-|\alpha|}{2}}=\underbrace{\Big(\frac{|\alpha|!|\beta-\tilde{\delta}|!}{|\alpha-\delta|!|\beta|!}\Big)^{1/2}}_{\leq 1}\underbrace{\Big(\frac{|\beta-\tilde{\delta}|!}{|\alpha-\delta|!}|\beta|^{-|\beta|+|\alpha|}\Big)^{1/2}}_{\leq 1}\leq 1.
\end{equation}

By \eqref{c1bis}, \eqref{c3bis}, \eqref{c4bis2}, we get

\begin{multline}\label{c6bis}
\frac{\eps^{|\alpha|+|\beta|}}{|\alpha|!^{1/2}|\beta|!^{1/2}}
|C_{\alpha,\beta,\delta,\ga_0,\ga_1,\ldots,\ga_j}|
\\ \times \|E\big((1-\varphi_\beta(x))p_{\alpha,\beta,\delta,\ga_0,\ga_1,\ldots,\ga_j}(x,D)(
x^{\beta-\tid-\ga_0-\ldots-\ga_j}\partial^{\alpha-\delta-\tg_0-\ldots-\tg_j}
u)\big)\|_{Q^s}\\
\leq
C_s(C_s\epsilon)^{|\delta|+|\tilde{\delta}|+|\ga_0+\ldots+\ga_j|+|\tg_0+\ldots+\tg_{j}|}
\frac{\eps^{|\alpha|+|\beta|-|\delta|-|\tilde{\delta}|
-|\ga_0+\ldots+\ga_j|-|\tg_0+\ldots+\tg_{j}|}}
{|\beta-\tid-\ga_0-\ldots-\ga_j|!}\\
\times
\|x^{\beta-\tid-\ga_0-\ldots-\ga_j}\partial^{\alpha-\delta-\tg_0-\ldots-\tg_j}
u\|_{Q^s}.
\end{multline}
Since $|\beta|>|\alpha|$, we have $M(\alpha,\beta)=|\alpha|!^{1/2}|\beta|!^{1/2}$, whereas from \eqref{bede} we see that $M(\alpha-\delta-\tg_0-\ldots-\tg_j,\beta-\tid-\ga_0-\ldots-\ga_j)=|\beta-\tid-\ga_0-\ldots-\ga_j|!$. Hence we deduce from \eqref{c6bis} that
\begin{multline}\label{c5ter}
\frac{\eps^{|\alpha|+|\beta|}}{M(\alpha,\beta)}
|C_{\alpha,\beta,\delta,\ga_0,\ga_1,\ldots,\ga_j}|
\\ \times \|E\big((1-\varphi_\beta(x))p_{\alpha,\beta,\delta,\ga_0,\ga_1,\ldots,\ga_j}(x,D)(
x^{\beta-\tid-\ga_0-\ldots-\ga_j}\partial^{\alpha-\delta-\tg_0-\ldots-\tg_j}
u)\big)\|_{Q^s}\\
\leq
C_s(C_s\epsilon)^{|\delta|+|\tilde{\delta}|+|\ga_0+\ldots+\ga_j|+|\tg_0+\ldots+\tg_{j}|}
\frac{\eps^{|\alpha|+|\beta|-|\delta|-|\tilde{\delta}|
-|\ga_0+\ldots+\ga_j|-|\tg_0+\ldots+\tg_{j}|}}
{M(\alpha-\delta-\tg_0-\ldots-\tg_j,\beta-\tid-\ga_0-\ldots-\ga_j)}\\
\times
\|x^{\beta-\tid-\ga_0-\ldots-\ga_j}\partial^{\alpha-\delta-\tg_0-\ldots-\tg_j}
u\|_{Q^s}.
\end{multline}

We now use \eqref{c00bis}, \eqref{c000bis},
\eqref{c0000bis} \eqref{c6bisbis2}, \eqref{c5ter} to
conclude, by the same arguments as in the case
$|\alpha|\geq|\beta|$, that

\begin{equation}\label{ultima1}
\sum_{|\alpha|+|\beta|\leq N\atop |\beta|>|\alpha|}\frac{\eab}{\mab}\|E\big((1-\varphi_\beta(x))[P,x^\beta\partial^\alpha]u\big)\|_{Q^s} \leq C'_s\eps S^{s,\eps}_{N-1}[u].
\end{equation}

\par\medskip

 {\bf Estimate of $\displaystyle \frac{\eab}{\mab}\|E\big(\varphi_\beta(x)[P,x^\beta
\partial^\alpha]u\big)\|_{Q^s}
$}.\par We now start from the formula \eqref{c0}. For any
fixed $\alpha,\beta,\delta$, we choose
$\tilde{\delta}\leq\beta$ such that
$|\beta-\tilde{\delta}|=|\alpha-\delta|$, which is possible
because here $|\beta|>|\alpha|$. Writing $x^\beta
=x^{\tilde{\delta}}x^{\beta-\tilde{\delta}}$ in \eqref{c0}
and using \eqref{pr1} we still get the formula
\eqref{c0bis} (notice however the choice of
$\tilde{\delta}$ is different from the one we made there).
We now apply the iterative argument detailed at the
beginning of the present proof, which led to \eqref{c0ter},
where the coefficients
$C_{\alpha,\beta,\delta,\ga_0,\ga_1,\ldots,\ga_j}$ satisfy
\eqref{c1bis} and \eqref{ordinecomm} holds. The
multi-indices $\tg_j$, $j\geq0$, are chosen here to
satisfy, in addition, $|\tg_j|=|\gamma_j|$, which is
possible because $|\beta-\tilde{\delta}|=|\alpha-\delta|$.
Hence,  we can rewrite \eqref{c1bis} as
\begin{equation}\label{c1bis19p}
|C_{\alpha,\beta,\delta,\ga_0,\ga_1,\ldots,\ga_j}|\leq
\frac{|\alpha|!}
{\delta!\ga_0!|\beta-\tid-\ga_0-\ldots-\ga_j|!}
2^{|\tg_0+\ldots+\tg_{j-1}|}.
\end{equation}
Now, on the support of $\varphi_\beta$ we have  $|x|\leq 2|\beta|^{1/2}$; moreover we have $|\tilde{\delta}|=|\beta|-|\alpha|+|\delta|$. Hence it follows from \eqref{Gsymbols} and \eqref{ordinecomm} that

$$|\partial^\theta_\xi\partial^\sigma_x
\big(\varphi_\beta (x) p_{\alpha,\beta,\delta,\ga_0,\ga_1,\ldots,\ga_j}(x,\xi)\big)|\leq
C^{|\ga_0|+|\tilde{\delta}|+1}\ga_0!\delta!|\beta|^{\frac{|\beta|-|\alpha|}{2}}(1+|x|+|\xi|)^{m-|\theta|}\langle x\rangle^{-|\sigma|}
$$
for some constant $C$ depending on $\sigma$, $\theta$. As a consequence,
\begin{equation}\label{c3bis19p}
 \|E\circ (\varphi_\beta(x)p_{\alpha,\beta,\delta,\ga_0,\ga_1,\ldots,\ga_j}(x,D))
 \|_{\mathcal{B}(Q^s)}\leq
 C_s^{|\ga_0|+|\tilde{\delta}|+1}\ga_0!\delta! |\beta|^{\frac{|\beta|-|\alpha|}{2}}.
\end{equation}
Now we see from Stirling's formula that, for some $C>1$,
\begin{equation}\label{c4bis19p}
\frac{|\alpha|!^{1/2}|\beta|^{\frac{|\beta|-|\alpha|}{2}}}{|\beta|!^{1/2}}\leq C^{|\beta|-|\alpha|}\leq C^{|\tilde{\delta}|}.
\end{equation}
By applying \eqref{c1bis19p}, \eqref{c3bis19p} and \eqref{c4bis19p} we obtain
\begin{multline}\label{c6bis19p}
\frac{\eps^{|\alpha|+|\beta|}}{|\alpha|!^{1/2}|\beta|!^{1/2}}
|C_{\alpha,\beta,\delta,\ga_0,\ga_1,\ldots,\ga_j}|
\\ \times \|E\big(\varphi_\beta(x)p_{\alpha,\beta,\delta,\ga_0,\ga_1,\ldots,\ga_j}(x,D)(
x^{\beta-\tid-\ga_0-\ldots-\ga_j}\partial^{\alpha-\delta-\tg_0-\ldots-\tg_j}
u)\big)\|_{Q^s}\\
\leq
C_s(C_s\epsilon)^{|\delta|+|\tilde{\delta}|+|\ga_0+\ldots+\ga_j|+|\tg_0+\ldots+\tg_{j}|}
\frac{\eps^{|\alpha|+|\beta|-|\delta|-|\tilde{\delta}|
-|\ga_0+\ldots+\ga_j|-|\tg_0+\ldots+\tg_{j}|}}
{|\beta-\tid-\ga_0-\ldots-\ga_j|!}\\
\times
\|x^{\beta-\tid-\ga_0-\ldots-\ga_j}\partial^{\alpha-\delta-\tg_0-\ldots-\tg_j}
u\|_{Q^s}.
\end{multline}
Since $|\beta|>|\alpha|$ we have $M(\alpha,\beta)=|\alpha|!^{1/2}|\beta|!^{1/2}$,  whereas our choice of $\tilde{\delta}$ and $\tg_j$ implies that $|\alpha-\delta-\tg_0-\ldots-\tg_j|=|\beta-\tid-\ga_0-\ldots-\ga_j| $. Then we have $M(\alpha-\delta-\tg_0-\ldots-\tg_j,\beta-\tid-\ga_0-\ldots-\ga_j)=|\beta-\tid-\ga_0-\ldots-\ga_j|!$ and we can rewrite \eqref{c6bis19p} as

\begin{multline}\label{c5ter19p}
\frac{\eps^{|\alpha|+|\beta|}}{\mab}
|C_{\alpha,\beta,\delta,\ga_0,\ga_1,\ldots,\ga_j}|
\\ \times \|E\big(\varphi_\beta(x)p_{\alpha,\beta,\delta,\ga_0,\ga_1,\ldots,\ga_j}(x,D)(
x^{\beta-\tid-\ga_0-\ldots-\ga_j}\partial^{\alpha-\delta-\tg_0-\ldots-\tg_j}
u)\big)\|_{Q^s}\\
\leq
C_s(C_s\epsilon)^{|\delta|+|\tilde{\delta}|+|\ga_0+\ldots+\ga_j|+|\tg_0+\ldots+\tg_{j}|}
\frac{\eps^{|\alpha|+|\beta|-|\delta|-|\tilde{\delta}|
-|\ga_0+\ldots+\ga_j|-|\tg_0+\ldots+\tg_{j}|}}
{M(\alpha-\delta-\tg_0-\ldots-\tg_j,\beta-\tid-\ga_0-\ldots-\ga_j)}\\
\times
\|x^{\beta-\tid-\ga_0-\ldots-\ga_j}\partial^{\alpha-\delta-\tg_0-\ldots-\tg_j}
u\|_{Q^s}.
\end{multline}

It follows from \eqref{c0ter} and \eqref{c5ter19p}, by the same arguments\footnote{To be precise, here we do not have longer $|\tilde{\delta}|\leq|\delta|$, but rather $|\tilde{\delta}|=|\beta|-|\alpha|+|\delta|$, so that the number of multi-indices $\tilde{\delta}$ which may arise is estimated by $2^{|\beta|-|\alpha|+|\delta|+d-1}$, and this factor is absorbed by the power $\epsilon^{|\tilde{\delta}|}=\epsilon^{|\beta|-|\alpha|+|\delta|}$ in \eqref{c5ter19p}.} as in the case $|\alpha|\geq|\beta|$, that

\begin{equation}
\sum_{|\alpha|+|\beta|\leq N\atop |\beta|>|\alpha|}\frac{\eab}{\mab}\|E\big(\varphi_\beta(x)[P,x^\beta\partial^\alpha]u\big)\|_{Q^s} \leq C'_s\eps
 S^{s,\eps}_{N-1}[u].
\end{equation}
This estimate, together with \eqref{ultima0} and \eqref{ultima1}, implies \eqref{nn5}, which concludes the proof.
\end{proof}

%%%%%%%%%%%%%%%%%%%%%%%%%%%%%%%%%%%%%%%%%%%%%%%%%%%%%%%%%%%%%%%%%%%%%%%%%%%%%%%%%%%%%%%%%%%%%%%%%%%%%%%%%%%%%%%%%%%%%%%%%%%%%%%%%%%%%%%%%%%%%%%%%

We now turn the attention to
the nonlinear term. We first treat the case when $m \geq 1.$
\begin{proposition}\label{nonlin1}
Let $E\in{\rm OP}\Gamma^{-m}(\rd)$, $m\geq1, h \in \N,
\rho_{1}, \ldots, \rho_{l}\in\mathbb{N}^d$, with
$h+\max\{|\rho_{k}|\}\leq m-1.$ Let $g$ be a real-analytic
function on $\R^{d}$ satisfying the estimates
\begin{equation}\label{coeffnonlin}
|\partial^{\alpha}g(x)|\leq C^{|\alpha|+1}\alpha!\px^{h-|\alpha|}, \qquad x \in \rd, \alpha \in \N^{d}
\end{equation}
for some $C>0$ independent on $\alpha$. Then for every integer
$s>d/2+\max_k\{|\rho_k|\}$
there exists a constant
$C_s>0$ such that, for every
$\eps$ small enough, $N \geq 1$ and $u\in
\cS(\rd)$, the following
estimates hold:
\begin{equation}\label{c4tris}
\sum_{0<|\alpha|+|\beta|\leq
N}\frac{\eab}{\mab}\|E\big(
x^\beta\partial^\alpha\big(g(x) \prod_{k=1}^l \partial^{\rho_k}u\big)\big)\|_{Q^s}
\leq C_s \ve(S_{N-1}^{s,\eps}[u])^{l}.
\end{equation}
\end{proposition}
\begin{proof}

We first treat the terms with $\beta\not=0$ in the sum
\eqref{c4tris}. Let
$j\in\{1,\ldots,d\}$ such
that $\beta_{j}\not=0$. By Leibniz' formula, we have
\[
x^\beta\partial^\alpha\big(g(x)
\prod_{k=1}^l
\partial^{\rho_k} u\big)=x_{j}
\sum_{\delta_0\leq \alpha-e_j}\sum_{\delta_1+\ldots+\delta_{l}=\alpha-\delta_0}
\frac{\alpha!}{\delta_0!\delta_1!
\ldots\delta_{l}!}\partial^{\delta_{0}}g(x)
x^{\beta-e_j}\prod_{k=1}^l
\partial^{\delta_k+\rho_k} u.
\]
 Let now $\tilde{\delta}_0$ be a
 multi-index of maximal length among
 those satisfying
 $|\tilde{\delta}_0|\leq|\delta_0|$
 and $\tilde{\delta}_0\leq\beta-e_j$. Write $x^{\beta-e_j}=x^{\tilde{\delta}_0} x^{\beta-e_j-\tilde{\delta}_0}$ and
 observe that the symbol $x_jx^{\tilde{\delta}_0}\partial^{\delta_{0}}g(x)$ belongs to $\Gamma^{h+1}(\rd)$, with every seminorm estimated by $A^{|\delta_0|+1}\delta_0!$ for some positive constant $A$ independent of $\delta_0$. Then
 $E\circ x_{j}
  x^{\tilde{\delta}_0}\partial^{\delta_{0}}g(x)$ belongs to ${\rm OP}\Gamma^{-m+1+h}(\rd)$. Consequently, it is continuous $
 Q^{s-M}(\rd)\to Q^{s}(\rd)$, with $M=\max\{|\rho_k|\}$, since $-m+1+h\leq -M$ and its operator norm is bounded by $A^{|\delta_0|+1}\delta_0!$ for a new constant $A$ independent of $\delta_0$. Then we have
\begin{multline*}
\frac{\eab}{\mab}\|E\big(
x^\beta\partial^\alpha\big(g(x)
\prod_{k=1}^l
\partial^{\rho_k} u\big)\big)\|_{Q^s}\\
 \leq C_s \sum_{\delta_0\leq \alpha}A^{|\delta_0|+1}\sum_{\delta_1+\ldots+\delta_{l}=\alpha-\delta_0}
\frac{\eab}{\mab}\frac{\alpha!}{\delta_1!
\ldots\delta_{l}!}
 \|x^{\beta-e_j-\tilde{\delta}_0}\prod_{k=1}^l
\partial^{\delta_k+\rho_k} u\|_{Q^{s-M}}.
\end{multline*}
We can now write
$$
x^{\beta-e_j-\tilde{\delta}_0}\prod_{k=1}^l
\partial^{\delta_k+\rho_k} u
=\prod_{k=1}^l x^{\ga_k}
\partial^{\delta_k+\rho_k} u,
$$
where
$\ga_1+\ldots+\ga_{l}=\beta-e_j-\tilde{\delta}_0$
and, if $|\beta|\leq
|\alpha|$, with
 $|\ga_k|\leq|\delta_k|$ for $1\leq
k\leq l$ (which is possible because in
that case
$|\beta-e_j-\tilde{\delta}_0|\leq|\alpha-\delta_0|$;
observe that if
$|\tilde{\delta}_0|<|\delta_0|$ then
$\beta-e_j-\tilde{\delta}_0=0$), whereas,
if $|\beta|\geq |\alpha|+1$, with
$|\ga_k|\geq|\delta_k|$ for $1\leq
k\leq l$ (which is possible because in
that case
$|\tilde{\delta}_0|=|\delta_0|$ and
$|\beta-e_j-\tilde{\delta}_0|\geq|\alpha-\delta_0|$). Moreover, if $|\beta|\leq|\alpha|,$ (then $M(\alpha, \beta)=|\alpha|!$), we have by \eqref{pr8} the following inequality
\begin{equation}\label{nn3}
\frac{1}{|\alpha|!}\cdot
\frac{\alpha!}
{\delta_1!\ldots\delta_{l}!}
\leq\frac{1}{|\delta_1|!\ldots|\delta_{l}|!},
\end{equation}
whereas, for $|\beta|\geq|\alpha|+1,$ (then $M(\alpha, \beta)=|\alpha|!^{1/2}|\beta|!^{1/2}$), we have
\begin{equation}\label{nn4}
\frac{1}{|\alpha|!^{1/2}|\beta|!^{1/2}}\cdot
\frac{\alpha!}
{\delta_{1}!\ldots\delta_{l}!}
\leq\frac{1}{(|\delta_{1}|!\ldots|\delta_{l}|!|\ga_1|!
\ldots|\ga_{l}|!)^{1/2}},
\end{equation}
which also follows at once from
\eqref{pr8}.
Hence by Proposition \ref{schauderQ} we get
 \begin{multline}
\frac{\eab}{\mab}\|E\big(
x^\beta\partial^\alpha\big(g(x)
\prod_{k=1}^l
\partial^{\rho_k} u\big)\big)\|_{Q^s}
\\ \leq C_s\eps
\sum_{\delta_0\leq \alpha}(A\ve)^{|\delta_0|}\ve^{|\tilde{\delta}_0|}\sum_{\delta_1+\ldots+\delta_{l} = \alpha-\delta_0}
\prod_{k=1}^{l}
\frac{\eps^{|\gamma_k|+|\delta_k|}}
{M(\gamma_k,\delta_k)}\|x^{\ga_k}
\partial^{\delta_k+\rho_k}
u\|_{Q^{s-M}}. \label{nn1bis}
\end{multline}
 Let now $T\in{\rm OP}\Gamma^{-M}(\rd)$ be any operator which gives  an isomorphism $Q^{s-M}\to Q^{s}$, and write
$x^{\gamma_k}\partial^{\delta_k+\rho_k}u=\partial^{\rho_k}
\big(x^{\gamma_k}\partial^{\delta_k}
u\big)+[x^{\gamma_k}\partial^{\delta_k},\partial^{\rho_k}]u$
in the last term of \eqref{nn1bis}. We get
\[
\|x^{\gamma_k}\partial^{\delta_k+\rho_k}u\|_{Q^{s-M}}
\leq
\|x^{\gamma_k}\partial^{\delta_k}
u\|_{Q^s}+\| T [x^{\gamma_k}
\partial^{\delta_k},\partial^{\rho_k}]u\|_{Q^s},
\]
where we used the fact that $\partial^{\rho_k}$ is bounded $Q^s(\rd)\to Q^{s-M}(\rd)$.

Using this last estimate we
obtain
\begin{multline*}
\frac{\eab}{\mab}\|E\big(
x^\beta\partial^\alpha\big(g(x)
\prod_{k=1}^l
\partial^{\rho_k} u\big)\big)\|_{Q^s}
\leq  C_s\eps \sum_{\delta_0 \leq
\alpha}(A\ve)^{|\delta_0|}\hskip-10pt\sum_{
\tilde{\delta}_0\leq \beta-e_j}\hskip-7pt
\ve^{|\tilde{\delta}_0|} \hskip-7pt \\ \times
\sum_{\delta_1+\ldots+\delta_{l}=\alpha-\delta_0}\prod_{k=1}^{l}\frac{\eps^{|\gamma_k|+|\delta_k|}}
{M(\gamma_k,\delta_k)}  \Big\{
\|x^{\ga_k}\partial^{\delta_k} u\|_{Q^s}+\sum_{|\gamma|\leq
m-1}\|T[x^{\gamma_k}
\partial^{\delta_k},\partial^\gamma]u\|_{Q^s}\Big\},
\end{multline*}
(recall that the $\gamma_k$'s depend on
$\alpha,\beta,\tilde{\delta}_0,\delta_1,\ldots,\delta_{l}$ and the choice of $e_j$). We
now sum the above expression over
$|\alpha|+|\beta|\leq N$,
$\alpha\not=0$.  When $\alpha$ and
$\beta$ vary but $\delta$ and $\tilde{\delta_0}$ are fixed, every term in the above
sum also appears in the development
of
\[
\Big\{\sum_{|\tilde{\alpha}|+|\tilde{\beta}|\leq
N-1}\frac{\eps^{|\tilde{\alpha}|+|\tilde{\beta}|}}
{M(\tilde{\alpha},\tilde{\beta})}\Big\{
\|x^{\tilde{\beta}}\partial^{\tilde{\alpha}}
u\|_{Q^s}+\sum_{|\gamma|\leq
m-1}\| T[x^{\tilde{\beta}}
\partial^{\tilde{\alpha}},\partial^\gamma]u\|_{Q^s} \Big\}\Big\}^{l}.
\]
and is repeated at most $d$ times (corresponding to the possibile choices of $e_j$).
Hence, taking $\ve$ sufficiently small, we
obtain
\begin{align*}
&\sum_{0<|\alpha|+|\beta|\leq N \atop \beta\not=0}\frac{\eab}{\mab}\|E\big(
x^\beta\partial^\alpha\big(g(x)
\prod_{k=1}^l
\partial^{\rho_k} u\big)\big)\|_{Q^s}
\\
&\leq
C''_s\eps\Big\{\sum_{|\tilde{\alpha}|+|\tilde{\beta}|\leq
N-1}
\frac{\eps^{|\tilde{\alpha}|+|\tilde{\beta}|}}
{M(\tilde{\alpha},\tilde{\beta})}\Big\{
\|x^{\tilde{\beta}}\partial^{\tilde{\alpha}}
u\|_{Q^s}+\sum_{|\gamma|\leq
m-1} \|T [x^{\tilde{\beta}}
\partial^{\tilde{\alpha}},\partial^\gamma]u\|_{Q^s} \Big\}\Big\}^{l}\\
&\leq C''_s\eps\big\{
S^{s,\eps}_{N-1}[u]+C'''_s\eps
S^{s,\eps}_{N-2}[u]\big\}^{l}\leq
C''''_s\eps
(S^{s,\eps}_{N-1}[u])^{l},
\end{align*}
where we used Proposition \ref{commutatore} applied with
 $\partial^\gamma$ and $T$ in place
 of $P$ and $E$ respectively, and we understand $S^{s,\eps}_{-1}[u]=0$.\par
  We now treat the
terms with $\beta=0$ in the sum \eqref{c4tris} (recall,
$M(\alpha,0)=|\alpha|!$). Let $\alpha\not=0$ and
$j\in\{1,\ldots,d\}$ such that $\alpha_{j}\not=0$. By
writing $\partial^\alpha=\partial_j\partial^{\alpha-e_j}$
and using Leibniz' formula we have
\[
\partial^\alpha\big(g(x)
\prod_{k=1}^l
\partial^{\rho_k} u\big)=\partial_{j}
\sum_{\delta_0\leq \alpha-e_j}\sum_{\delta_1+\ldots+\delta_{l}=\alpha-e_{j}-\delta_0}
\frac{(\alpha-e_{j})!}{\delta_0!\delta_1!
\ldots\delta_{l}!}\partial^{\delta_{0}}g(x)\prod_{k=1}^l
\partial^{\delta_k+\rho_k} u.
\]
 Observe that
 $E\partial_{j}\circ\partial^{\delta_{0}}g(x)\in{\rm OP}\Gamma^{-m+1+h}(\rd)$ is bounded $
 Q^{s-M}(\rd)\to Q^{s}(\rd)$, with $M=\max\{|\rho_k|\}$, because $-m+1+h\leq -M$, and its operator norm is estimated by $A^{|\delta_0|+1}\delta_0!$ for some positive constant $A$ independent of $\delta_0$. Hence
\begin{multline*}
\frac{\epsilon^{|\alpha|}}{|\alpha|!} \|E\partial^\alpha\big(g(x)
\prod_{k=1}^l
\partial^{\rho_k} u\big)\big)\|_{Q^s}\\
 \leq C_s \sum_{\delta_0\leq \alpha-e_j}A^{|\delta_0|+1}\sum_{\delta_1+\ldots+\delta_{l}=\alpha-e_{j}-\delta_0}
\frac{\epsilon^{|\alpha|}}{|\alpha|!}\frac{(\alpha-e_{j})!}{\delta_1!
\ldots\delta_{l}!}
 \|\prod_{k=1}^l
\partial^{\delta_k+\rho_k} u\|_{Q^{s-M}}.
\end{multline*}

 Using the inequality
\begin{equation*}
\frac{1}{|\alpha|!}\cdot
\frac{(\alpha-e_{j})!}
{\delta_1!\ldots\delta_{l}!}
\leq\frac{1}{|\delta_1|!\ldots|\delta_{l}|!},
\end{equation*}
 Proposition \ref{schauderQ} and the boundedness of $\partial^{\rho_k}: Q^{s}\to Q^{s-M}$ we get
%\begin{multline}
$$\frac{\epsilon^{|\alpha|}}{|\alpha|!} \|E\partial^\alpha\big(g(x)
\prod_{k=1}^l
\partial^{\rho_k} u\big)\big)\|_{Q^s}\leq C_s\eps
\sum_{\delta_0\leq \alpha-e_j}(A\ve)^{|\delta_0|}\hskip-12pt \sum_{\delta_1+\ldots+\delta_{l}
= \alpha-e_{j}-\delta_0}
\prod_{k=1}^{l}
\frac{\eps^{|\delta_k|}}
{|\delta_k|!}\|\nonumber
\partial^{\delta_k}
u\|_{Q^{s}}. $$
By the same arguments as above we obtain
\[
\sum_{0<|\alpha|\leq
N}\frac{\epsilon^{|\alpha|}}{|\alpha|!}\|E\big(
x^\beta\partial^\alpha\big(g(x) \prod_{k=1}^l \partial^{\rho_k}u\big)\big)\|_{Q^s}
\leq C_s \ve(S_{N-1}^{s,\eps}[u])^{l},
\]
which concludes the proof.
\end{proof}

We are now ready to conclude the proof of Theorem
\ref{AA5.1}. \\

\noindent
\textit{End of the proof of
Theorem \ref{AA5.1} (the case $m\geq 1$).} From \eqref{A5.1f} we have, for
$\alpha,\beta\in\N^d$,
$\epsilon>0$,
\[
\frac{\eab}{\mab}
x^\beta\partial^\alpha
Pu=\frac{\eab}{\mab}
x^\beta\partial^\alpha
 f +
 \frac{\eab}{\mab}
x^\beta\partial^\alpha F[u],
\]
so that
%\begin{multline*}
$$\frac{\eab}{\mab}P(
x^\beta\partial^\alpha
u)=\frac{\eab}{\mab}[P,x^\beta\partial^\alpha]u+\frac{\eab}{\mab}
x^\beta\partial^\alpha
 f +
 \frac{\eab}{\mab}
x^\beta\partial^\alpha F[u].$$
%\end{multline*}
We now apply to both sides
the parametrix $E$ of $P$.
With $R=EP-I\in{\rm
OP}\Gamma^{-1}(\rd)$ we obtain
\begin{multline*}
\frac{\eab}{\mab}x^\beta\partial^\alpha
u=-\frac{\eab}{\mab}
R(x^\beta\partial^\alpha
u)+\frac{\eab}{\mab}
E[P,x^\beta\partial^\alpha]u\\
+\frac{\eab}{\mab}E(x^\beta\partial^\alpha
f)+
%\sum_{h,l}\sum_{\rho_1,\ldots,\rho_l}
%F_{h,l,\rho_1,\ldots,\rho_l}
\frac{\eab}{\mab}E(x^\beta\partial^\alpha F[u] ).
%\Big(x^h
%\prod_{k=1}^l
%\partial^{\rho_k}u\Big)\Big).
\end{multline*}
Taking the $Q^s$ norms and
summing over
$|\alpha|+|\beta|\leq N$ give
\begin{eqnarray}\label{colonna}
S^{s,\epsilon}_N[u] &\leq&
\|u\|_{Q^s}+\sum_{0<|\alpha|+|\beta|\leq
N}\frac{\eab}{\mab}\|R(x^\beta\partial^\alpha
u)\|_{Q^s}  \\ \nonumber
&&+\sum_{0<|\alpha|+|\beta|\leq
N}\frac{\eab}{\mab}\|E[P,x^\beta\partial^\alpha]u\|_{Q^s}
\\ &&+\sum_{0<|\alpha|+|\beta|\leq
N}\frac{\eab}{\mab}\|E(x^\beta
\partial^\alpha f)\|_{Q^s} \nonumber \\
&&+\sum_{0<|\alpha|+|\beta|\leq
N}\frac{\eab}{\mab}\|E(x^\beta\partial^\alpha F[u] )\|_{Q^s}. \nonumber
\end{eqnarray}
The second and the third term in the right-hand side of \eqref{colonna} can be estimated using Propositions \ref{stima1} and \ref{commutatore} while the term containing $f$ is obviously dominated by $S_{\infty}^{s,\ve}[f].$ For the last term we can apply Proposition \ref{nonlin1}. Hence,  we have that, for
$\eps$ small enough,
$$S^{s,\epsilon}_N[u]\leq
\|u\|_{Q^s}+ C_s
S^{s,\epsilon}_\infty[f]+C_s\epsilon\Big(
S^{s,\epsilon}_{N-1}[u]+
\sum_{l}(S^{s,\eps}_{N-1}[u])^{l}\Big).$$
Iterating the last estimate and possibly shrinking $\ve$, we obtain that $S^{s,\eps}_\infty[u]<\infty,$
which implies
$u\in\sg$ by Proposition
\ref{stima0}. \par \qed

%---------------------------------------------------------------------------------------------------------------------------------------
%---------------------------------------------------------------------------------------------------------------------------------------

\subsection{Proof of Theorem \ref{AA5.1}: the case $0 < m < 1$}
In this case the nonlinearity \eqref{A5.2}, due to the restriction $h+\max\{|\rho_{k}|\} \leq \max\{m-1, 0 \}$ reduces to the following form
\begin{equation}
\label{A5.2tris} F[u]=\sum_{l}
F_{l}(x) u^l,
\end{equation}
the above sum being finite, with $l\in\N$, $l\geq2$, and
$F_{l}(x)$ real-analytic functions satisfying the following
estimates
\begin{equation}\label{coeffnonlinm<1}|\partial^{\alpha}F_{l}(x)|\leq C^{|\alpha|+1}\alpha!\px^{-|\alpha|}, \qquad x \in \R^{d}, \alpha \in \N^{d},
\end{equation}
for some $C>0$ independent of $\alpha.$
\par We follow the
same argument used for the case $m\geq 1$, so that we only
sketch the proof. For technical reasons which will be clear in the sequel, here it is convenient to work in the
framework of the usual Sobolev spaces, i.e. by defining
\begin{equation}\label{25-4}
\widetilde{S}_N^{s,\eps}[f]=\sum_{|\al|+|\beta|\leq N}\frac{\eab}{\mab}
\|x^\be\partial^\al f\|_{H^s},\quad
\widetilde{S}_\infty^{s,\eps}[f]=\sum_{\al,\,\beta \in \N^d}\frac{\eab}{\mab}
\|x^\be\partial^\al f\|_{H^s}. \end{equation} It is easy to
see that the results in Propositions \ref{stima0},
\ref{stima1} and \ref{commutatore} continue to hold with
$S_N^{s,\ve}[u]$ and $S_{\infty}^{s,\ve}[u]$ replaced by $\widetilde{S}_N^{s,\ve}[u]$ and $\widetilde{S}_{\infty}^{s,\ve}[u]$ and with the spaces $Q^s$
replaced by $H^s$ everywhere (operators in ${\rm
OP}\Gamma^0(\rd)$ are bounded on every $H^s$ by Proposition
\ref{25pro} with $m=n=0$). It remains to estimate the
nonlinear term. On this point we observe that although this
term is more elementary than before, the action of the
parametrix gives a lower ``gain'', since
$0<m<1.$ Then we have to modify slightly our technique. We
have the following result.

\begin{proposition}
\label{Gammanonlinearestimates} Let $E \in {\rm
OP}\Gamma^{-m}(\rd), 0<m<1,$ and let $l\in\mathbb{N}$,
  $l\geq2$ and $g$ be a real-analytic function on $\R^{d}$ satisfying the same
  estimates
  as in \eqref{coeffnonlinm<1}. Then, for every integer $s>d/2$
  there exists
a constant $C'_s>0$ and, for every
$\tau>0$, there exists
 $C_{\tau}>0$ such
 that, for every $\eps$ small
 enough, $N \geq 1$ and $u\in \cS(\rd)$
 we have
\begin{multline} \label{nonlinestimate2}
\sum_{0<|\alpha+\beta|\leq N}
\frac{\ve^{|\alpha|+|\beta|}}{M(\alpha,\beta)} \| E
(x^{\beta}\partial^{\alpha}(g(x)u^\ell)) \|_{H^s} \leq  \tau
C'_{s} \| u\|_{H^s}^{\ell-1} S_{N}^{s,\ve}[u]
\\+ C'_{s} (\ve
C_{\tau}+\tau+\ve) (S_{N-1}^{s,\ve}[u])^{\ell}.
\end{multline}
\end{proposition}

\begin{proof} We first consider the terms with $\beta \neq 0.$ We
can write $$x^{\beta}\partial^{\alpha}(g(x)u^{l})=g(x)x^{\beta}\partial^{\alpha}u^l +
\sum_{\delta_0+\delta_1+\ldots+\delta_k=\alpha\atop
\delta_0\not=0}
\frac{\alpha!}{\delta_0!\delta_1!\cdots\delta_l!}\partial^{\delta_0}g(x)
x^{\beta}\prod_{k=1}^l\partial^{\delta_k}u.$$

Let $\tilde{\delta}_0$ be a multi-index of maximal length
among those satisfying $|\tilde{\delta}_0|\leq|\delta_0|$.
$\tilde{\delta}_0\leq\beta$. The operators $E \circ g(x)$
and $E \circ x^{\tilde{\delta}_0}\partial^{\delta_0} g(x)$
belong to ${\rm OP}\Gamma^{-m}(\rd)$, and the symbol of the
second one has each seminorm estimated by
$A^{|\delta_0|+1}\delta_0!$, for some positive constant $A$
independent on $\delta_0$. Hence, by the continuity
properties on weighted Sobolev spaces (Proposition
\ref{25pro} with $n=0$) we have
 \begin{eqnarray*}
\|E(x^{\beta}\partial^{\alpha}(g(x)u^{l})\|_{H^s} &\leq& C_s
\|\px^{-m}x_jx^{\beta-e_j}\partial^{\alpha}(u^{l})\|_{H^s} \\
&&+\sum_{\delta_0+\delta_1+\ldots+\delta_k=\alpha\atop
\delta_0\not=0
}C_s^{|\delta_0|+1}\frac{\alpha!}{\delta_1!\cdots\delta_l!}\|
x^{\beta-\tilde{\delta}_0}\prod_{k=1}^l\partial^{\delta_k}u\|_{H^s}.
\end{eqnarray*}
Now, since $\beta_j \neq 0$ for some $j \in \{1,\ldots, d\}$, we have
\begin{eqnarray*}
\|\px^{-m}x_jx^{\beta-e_j}\partial^{\alpha}(u^{l})\|_{H^s} \hskip-2pt &=& \hskip-2pt \sum_{|\gamma|\leq s}\|\partial^{\gamma}(\px^{-m}x_j x^{\beta-e_j}\partial^{\alpha}(u^l))\|_{L^2}
\\ \hskip-2pt  &\leq& \hskip-2pt  \sum_{|\gamma|\leq s}\|\px^{-m}x_j\partial^{\gamma}(x^{\beta-e_j}\partial^{\alpha}(u^l))\|_{L^2} \\
&& \hskip-2pt +\sum_{|\gamma|\leq s}\sum_{0\neq \gamma'\leq
\gamma}\hskip-3pt \binom{\gamma}{\gamma'}
\|\partial^{\gamma'}(\px^{-m}x_j)\partial^{\gamma-\gamma'}
(x^{\beta-e_j}\partial^{\alpha}(u^l))\|_{L^2}.\end{eqnarray*}
Now, for every $\tau >0$ there exists $C'_{\tau}>0$ such
that \begin{equation}\label{2525}
 \px^{-m}|x_j| \leq \tau
|x_j|+C'_{\tau}.
\end{equation}
 Using
this inequality and commuting $x_j$ with
$\partial^{\gamma}$ we get
 \begin{multline*}\hskip-10pt \sum_{|\gamma|\leq s}\|\px^{-m}x_j \partial^{\gamma}(x^{\beta}\partial^{\alpha}(u^l))\|_{L^2} \leq  \tau \hskip-4pt \sum_{|\gamma|\leq s}\|x_j\partial^{\gamma}(x^{\beta-e_j}\partial^{\alpha}u^l)\|_{L^2}+C_{s,\tau}\|x^{\beta-e_j}\partial^{\alpha}(u^l)\|_{H^s} \\  \leq  \tau C_s \|x^{\beta}\partial^{\alpha}(u^l)\|_{H^s}+C'_{s,\tau}\|x^{\beta-e_j}\partial^{\alpha}(u^l)\|_{H^s}.\end{multline*}
We notice moreover that for $\gamma' \neq 0$ we have
$\partial^{\gamma'}(\px^{-m}x_j)\in L^{\infty}(\rd)$, so
that
$$\sum_{|\gamma|\leq s}\sum_{0\neq \gamma'\leq \gamma}\binom{\gamma}{\gamma'} \|\partial^{\gamma'}(\px^{-m}x_j)\partial^{\gamma-\gamma'}(x^{\beta-e_j}\partial^{\alpha}u^l)\|_{L^2} \leq C_s \|x^{\beta-e_j}\partial^{\alpha}u^l\|_{H^s}.$$
Hence we have obtained that \begin{eqnarray} \label{terna}
\hskip+22pt \|E(x^{\beta}\partial^{\alpha}(g(x)u^{l})\|_{H^s} &\leq& \tau C_s \|x^{\beta}\partial^{\alpha}(u^l)\|_{H^s}+C'_{s,\tau}\|x^{\beta-e_j}\partial^{\alpha}(u^l)\|_{H^s}
\\ &&+ \sum_{\delta_0+\delta_1+\ldots+\delta_k=\alpha\atop
\delta_0\not=0
}C_s^{|\delta_0|+1}\frac{\alpha!}{\delta_1!\cdots\delta_l!}\|
x^{\beta-\tilde{\delta}_0}\prod_{k=1}^l\partial^{\delta_k}u\|_{H^s}.
\nonumber
\end{eqnarray}
Let us estimate the three terms in the right-hand side of \eqref{terna}.
To treat the first one we observe that
$$x^{\beta}\partial^{\alpha}(u^{\ell})= \ell u^{\ell-1}x^{\beta}\partial^{\alpha}u  +
\sum_{\delta_1+\ldots+\delta_{\ell}=\alpha\atop\delta_{k}\neq
\alpha\, \forall k}\frac{\alpha!}{\delta_{1}! \ldots
\delta_{\ell}!}
\prod_{k=1}^lx^{\gamma_k}\partial^{\delta_{k}}u,$$ where,
as before, we can choose $\gamma_1,\ldots, \gamma_l \in
\N^d$ such that $\gamma_1+\ldots+\gamma_l=\beta$ and
$|\gamma_j| \leq |\delta_j|$ (respectively ($|\gamma_j|
\geq |\delta_j|$) if $|\beta| \leq |\alpha|$ (respectively
if $|\beta| \geq |\alpha|$). Then, using the same arguments
as in the case $m \geq 1,$ we obtain
\begin{equation} \label{terna1}\tau C_s  \sum_{\stackrel{|\alpha+\beta|\leq N}{\beta \neq 0}}\frac{\ve^{|\alpha+\beta|}}{M(\alpha, \beta)}  \|x^{\beta}\partial^{\alpha}(u^{\ell})\|_{H^s} \leq \tau lC_s \|u\|_{H^s}^{l-1}\widetilde{S}_{N}^{s,\ve}[u] +\tau C_s  (\widetilde{S}_{N-1}^{s,\ve}[u])^l.
\end{equation}
Similarly, we easily prove that
\begin{equation} \label{terna2}
\sum_{\stackrel{|\alpha+\beta|\leq N}{\beta \neq
0}}\frac{\ve^{|\alpha+\beta|}}{M(\alpha, \beta)}
\|x^{\beta-e_j}\partial^{\alpha}(u^{\ell})\|_{H^s} \leq C_s \ve
(\widetilde{S}_{N-1}^{s,\ve}[u])^l. \end{equation} Concerning the third
term in \eqref{terna} we can write
$x^{\beta-\tilde{\delta}_0}= \prod_{k=1}^l x^{\gamma_k}$,
where $\gamma_1,\ldots, \gamma_l$ satisfy
$\gamma_1+\ldots+\gamma_l=\beta-\tilde{\delta}_0$ and
$|\gamma_j| \leq |\delta_j|$ (respectively ($|\gamma_j|
\geq |\delta_j|$) if $|\beta| \leq |\alpha|$ (respectively
if $|\beta| \geq |\alpha|$). Then, the same arguments as in
the case $m \geq 1$ yield
\begin{equation}
\sum_{\stackrel{|\alpha+\beta|\leq N}{\beta \neq
0}}\frac{\ve^{|\alpha+\beta|}}{M(\alpha, \beta)}
\sum_{\delta_0+\delta_1+\ldots+\delta_k=\alpha\atop
\delta_0\not=0
}C^{|\delta_0|+1}\frac{\alpha!}{\delta_1!\cdots\delta_l!}\|
x^{\beta-\tilde{\delta}_0}\prod_{k=1}^l\partial^{\delta_k}u\|_{H^s}
\leq C_s \ve (\widetilde{S}_{N-1}^{s,\ve}[u])^l \end{equation}
for $\ve>0$ sufficiently small.
 The estimate of the terms in \eqref{nonlinestimate2} with $\beta=0$
 (hence $\alpha\not=0$) is very similar but
 easier, relying on the inequality
 \begin{equation}
 \pxi^{-m}|\xi_j| \leq \tau
|\xi_j|+C'_{\tau}.
\end{equation}
in place of \eqref{2525}. We omit the details for the sake
of brevity.
\end{proof}

\textit{End of the proof of Theorem \ref{AA5.1} (the case
$0<m<1$).} Using the same argument as in the case $m\geq
1$, by the variants of Propositions \ref{stima0},
\ref{stima1} and \ref{commutatore} with $\tilde{S}_N^{s,\eps}[f]$
and $\tilde{S}_\infty^{s,\eps}[f]$ defined in \eqref{25-4} in place of $S_N^{s,\eps}[f]$
and $S_\infty^{s,\eps}[f]$, and with
the spaces $Q^s$ replaced by $H^s$, and by Proposition
\ref{Gammanonlinearestimates} we obtain
\begin{multline*}
\widetilde{S}_{N}^{s,\ve}[u] \leq \|u \|_{H^s}+ C'_{s}
\widetilde{S}_{\infty}^{s,\ve}[f]+ C'_s\ve
\widetilde{S}_{N-1}^{s,\ve}[u]+\sum_l\Big(\tau C'_{s} \|
u\|_{H^s}^{\ell-1}
 \widetilde{S}_{N}^{s,\ve}[u]
\nonumber
\\+ C'_{s} (\ve
C_{\tau}+\tau+\ve) (\widetilde{S}_{N-1}^{s,\ve}[u])^{\ell}\Big)
\end{multline*}
for every $N\geq1$
 and $\eps$ small enough. Now, choosing
$\tau < (2\sum_l C'_{s}
\|u\|_{s}^{\ell-1})^{-1}$ we obtain
$$\widetilde{S}_{N}^{s,\ve}[u] \leq 2\|u
\|_{H^s}+ 2C'_{s}
\widetilde{S}_{\infty}^{s,\ve}[f]+
2C'_s\ve \widetilde{S}_{N-1}^{s,\ve}[u]+
\sum_l\Big( 2C'_{s} (\ve C_{\tau}+\tau+\ve)
(\widetilde{S}_{N-1}^{s,\ve}[u])^{\ell}\Big).
$$
 Then we can
iterate the last estimate
observing that, shrinking
$\tau$ and then $\ve$, the
quantity $\ve C_{\tau}+\tau+\ve$
can be taken arbitrarily
small. This gives
$\widetilde{S}^{s,\ve}_\infty[u]<\infty$
and therefore $u\in\sg$.

\section{Examples and concluding remarks}\label{esempi}
\subsection{Some remarks on the analyticity estimates}\label{esempi1} Let us say a few words on the estimates
\begin{equation}\label{5ma}
|\partial^\alpha f(x)|\leq
C^{|\alpha|+1}\alpha!\langle x\rangle^{-|\alpha|}\qquad \forall\alpha\in\N^d, x\in\rd
\end{equation}
which are assumed for the coefficients of the metric in
\eqref{intro3}, (cf. also the nonlinearity in
\eqref{intro5}, \eqref{intro6}, and \eqref{A5.2}
\eqref{estnonlincoeff}). \par Locally they are exactly the
usual estimates of real-analyticity. To better understand
the meaning of the decay for $|x|\rightarrow +\infty$, let
us consider the following important class of examples.
Consider a real-analytic function $f$ in $\rd$ satisfying,
in polar coordinates $r,\omega$, $r>0$,
$\omega\in\mathbb{S}^{d-1}$,
\[
f(r\omega)=h(r^{-1},\omega),\qquad {\rm for}\ r>r_0,\ \omega\in\mathbb{S}^{d-1},
\]
for some $r_0>0$, where $h$ is an analytic function on $[0,r_0^{-1})\times \mathbb{S}^{d-1}$, hence {\it analytic up to 0 in the first variable}. Let us verify that then $f$ satisfies the estimates \eqref{5ma}. \par
Clearly, it is sufficient to check the estimates \eqref{5ma} for large $|x|$. Now, by assumption we have
\[
f(x)=f(r\omega)=\sum_{k=0}^\infty \frac{1}{k!}\varphi_k(\omega)r^{-k}=\sum_{k=0}^\infty \tilde{\varphi}_k(x),\qquad |x|>r_0,
\]
where $\varphi_k$ are analytic functions on
$\mathbb{S}^{d-1}$, and $\tilde{\varphi}_k(x)\frac{1}{k!}\varphi_k(\omega)r^{-k}$. Observe that the
functions $\tilde{\varphi}_k(x)$ are real-analytic
functions for $x\not=0$ and positively homogeneous of
degree $-k$. Moreover, by the very definition of
$\varphi_k$, for every $\overline{x}$,
$|\overline{x}|>r_0$, we have the estimates
\begin{equation}\label{5ma2}
|\partial^\alpha \tilde{\varphi}_k(x)|\leq C^{|\alpha|+k+1}\alpha!,
\end{equation}
for some constant $C>0$, in some neighborhood of $\overline{x}$ (this is easily verified in polar coordinates and then one uses the analyticity of the change of variables). Hence, by compactness, the estimates  \eqref{5ma2} hold, say, for $|x|=2r_0$. By homogeneity we deduce that
\begin{equation}\label{5ma3}
|\partial^\alpha \tilde{\varphi}_k(x)|\leq C^{|\alpha|+k+1}\alpha!\left|\frac{x}{2r_0}\right|^{-k-|\alpha|},\qquad x\not=0.
\end{equation}
Hence we obtain
\[
|\partial^\alpha f(x)|\leq \sum_{k=0}^\infty
|\partial^\alpha \tilde{\varphi}_k(x)|\leq C
\alpha!\sum_{k=0}^\infty
\left|\frac{x}{2r_0C}\right|^{-k-|\alpha|}\leq
C(2r_0C)^{|\alpha|}\alpha!|x|^{-|\alpha|},
\]
if $|x|>4r_0C$. This concludes the proof of \eqref{5ma}.\par\medskip

As another remark, we observe that the estimates \eqref{5ma} are in fact equivalent to
requiring that $f(x)$ extends to a bounded holomorphic function
$f(x+iy)$ in a sector of the type \eqref{cla3intro}  (see e.g. \cite[Proposition 5.1]{cappiello-nicola}). This is very useful to check the estimates \eqref{5ma} in concrete situations, as we will see below.

\subsection{Metric Laplacians}
Consider a smooth  Riemannian metric $g_{jk}(x)$ in $\R^d$.
The corresponding Laplace-Beltrami operator has the form
\begin{align*}
\mathcal{L}u&=\sum_{j,k=1}^d\frac{1}{\sqrt{g(x)}}\partial_j\left(\sqrt{g(x)}g^{jk}(x)\partial_k
u\right)\\
&=\sum_{j,k=1}^d \left(g^{jk}(x)\partial_j\partial_k u-g^{jk}(x)\sum_{\ell=1}^d\Gamma^{\ell}_{jk}(x)\partial_{\ell} u\right),
\end{align*}
where $g^{jk}$ is the inverse matrix of $g_{jk}$,
$g={\rm det}\,(g_{jk})$, and the Christoffel symbols are defined by
\[
\Gamma^{\ell}_{ij}=\frac{1}{2}\sum_{k=1}^d g^{k\ell}\left(\partial_i g_{kj}+\partial_j g_{ik}-\partial_k g_{ij}\right).
\]
 Let us assume that the metric is
real-analytic and satisfies the estimates
\begin{equation}\label{4ma}
|\partial^\alpha g_{jk}(x)|\leq
C^{|\alpha|+1}\alpha!\langle x\rangle^{-|\alpha|},\qquad g(x)>C^{-1},
\end{equation}
for some $C>0$, and every $\alpha\in\N^d$, $x\in\rd$.
Then the matrix $g^{jk}$ satisfies \eqref{intro2} and the
estimates in \eqref{intro3}.  If in addition we consider $V(x)$ and $F[u]$ as in
\eqref{intro4}, \eqref{intro5}, then the equation
\begin{equation}\label{3ma}
-\mathcal{L}u+V(x)u-\lambda u=F[u],\qquad \lambda\in\C,
\end{equation}
is a special case of $\eqref{intro1}$. Hence,
by Theorem \ref{mainthm}, every solution $u\in H^s(\rd)$,
$s>d/2+1$, of \eqref{3ma}, extends to a holomorphic function
$u(x+iy)$ in the sector of $\mathbb{C}^d$
 in \eqref{cla3intro}, satisfying there the estimates
in \eqref{cla4intro} for some constants $C>0$, $c>0$.\par

  As a model for the above type of metrics, one may consider the
  hyperboloid $\R^{d+1}$:
  \[
S=\{(x,t)\in\rd\times\R:\ t=\sqrt{1+|x|^2}\},
 \]
 parametrized by $x\in\rd$.
The Riemannian metric induced on $S$ by the Euclidean one
then satisfies the estimates \eqref{4ma}, as one can easily verify. \par
More generally, we can consider real-analytic scattering metrics in $\rd$. They play an important role in geometric scattering theory, see \cite{melrose1,melrose3}, \cite[Chapter 6]{melrose2}, and have received particular attention in the last years (see e.g. \cite{hassel} and the references therein). Indeed, natural perturbations of the Euclidean metric fall in that category. \par A real-analytic metric in $\rd$ is of scattering type if for any coordinate chart $V\subset \mathbb{S}^{d-1}$ and for some $r_0>0$ it has the form \begin{equation}
h(r^{-1},\eta;dr,rd\eta),\qquad {\rm for}\ r>r_0,
\end{equation}
where $r=|x|$, $\eta=(\eta_1,\eta_2,\ldots,\eta_{d-1})$ are real-analytic coordinates on $V$, and $h$ is a positive definite quadratic form in the last couple of variables, whose coefficients are analytic functions on $[0,r_0^{-1})\times V$. Moreover one requires that $h(0,\eta;dr,d\eta)$ is positive-definite.\footnote{Of course, this is equivalent to saying that $h(0,\eta;dr,rd\eta)$ is positive definite for every $r>0$.} \hskip-4pt Notice that this metric approaches the conic metric $h(0,\eta;dr,rd\eta)$ as $r\to+\infty$, which explains the terminology, sometimes used in the literature, of {\it asymptotically conic} metric. Notice, by comparison, that the Euclidean metric $|dx|^2$ in polar coordinates reads in fact $dr^2+r^2h'$, where $h'$ is the usual metric on $\mathbb{S}^{d-1}$.\par
Now, using the remark in Subsection \ref{esempi1} one sees that, in Euclidan coordinates, such a metric satisfies the estimates \eqref{4ma}. In particular, the bound from below in \eqref{4ma} is satisfied because $h(0,\omega;dr,rd\omega)$ is positive definite, and this last metric in Euclidean coordinates has coefficients which are homogeneous functions of degree $0$ (in fact, each $\eta_j$ is a real-analytic function of $x$ on $\R_+\times V$, positively homogeneous of degree $0$).

\subsection{The linear case}

This above result for the equation \eqref{4ma} seems interesting even in the linear case
($F[u]=0$), namely for the eigenfunctions of
$-\mathcal{L}+V(x)$. That equation appears naturally, for example, when looking for standing wave solutions (i.e. solutions of the type $v(t,x)=e^{i\lambda t} u(x)$) of the Schr\"odinger equation $i\partial_t v-\mathcal{L}v+V(x)v=0$ for scattering metrics (cf.\ \cite{hassel}).  \par
In the linear case we can even assume
$u\in\cS'(\rd)$. In fact, the existence of a parametrix for
$-\mathcal{L}+V(x)$ (Proposition \ref{para}), shows
that such a solution is automatically in $\mathcal{S}(\rd)$. Moreover,
if $V(x)$ is in addition real-valued, we know e.g.\ from
\cite{hormander79bis}
 (see also \cite[Theorem 4.2.9]{nicola}) that the operator $-\mathcal{L}+V(x)$, regarded
 as a symmetric operator in $L^2(\rd,\sqrt{g}dx)$ with domain $\cS(\rd)$, is essentially self-adjoint and $L^2(\rd,\sqrt{g}dx)$ has an orthonormal basis made of eigenfunctions\footnote{Since the eigenvalues of the metric are bounded from below and from above, $L^2(\rd,dx)=L^2(\rd,\sqrt{g}dx)$ as normed space.}.
  Also, ${\rm dim\,Ker}\left( -\mathcal{L}+V(x)-\lambda\right)<\infty$, which implies that
  the width $\epsilon$ of the sector in \eqref{cla3intro} can then be chosen uniformly
  with respect to the solutions. \par\medskip

\subsection{Sharpness of Theorem \ref{mainthm}} We recall that in the case of a differential operator with polynomial coefficients the solutions $u \in \mathcal{S}'(\rd)$ of the equation $Pu=0$ extend to entire functions on $\C^d$ satisfying estimates \eqref{cla4intro} in a sector of the form \eqref{cla3intro}, cf. \cite[Theorem 1.1]{A18}. Very simple examples show that, even in the linear case, in Theorem \ref{mainthm} we cannot expect an entire extension for the solution $u$. For example, consider, in dimension $d=1$, the equation
\begin{equation}\label{6ma}
-u''+V(x)u=0,
\end{equation}
where $V(x)=x^2+3+\frac{2x^2-6}{(x^2+1)^2}$.
A solution is given by  $u(x)=\frac{1}{x^2+1} e^{-x^2/2}$, which does not extend to an entire function on $\C$.\par
The following example shows that, in fact, infinitely many singularities can occur along any fixed ray in the complex domain. Let $\theta\in(-\pi/2,\pi/2)$, and consider again the equation \eqref{6ma}, with
\[
V(x)=x^2-1-2e^{-2i\theta}+4e^{-i\theta}x\tanh(e^{-i\theta}x)+6e^{-2i\theta}x\tanh^2(e^{-i\theta}x).
\]
By applying the remark at the end of Subsection \ref{esempi1} to the function $\tanh(e^{-i\theta}x)$ it is immediate to check that $V(x)$ satisfies the estimates in \eqref{intro4}. On the other hand, the function
\[
u(x)= \cosh^{-2}(e^{-i\theta}x)e^{-x^2/2}
\]
is a solution of \eqref{6ma} and extends to a meromorphic
function in the complex plane with poles at
$z=e^{i(\theta+\pi/2)}(2k+1)\pi$, $k\in\mathbb{Z}$.\par
This shows that in Theorem \ref{mainthm}, even in the
linear case, the form of the domain of holomorphic
extension as a sector is sharp in general. The following
example shows a similar phenomenon in the nonlinear case,
even for the standard harmonic oscillator.\par

Consider the following nonlinear perturbation of the harmonic oscillator, in dimension $d=1$, at the first eigenvalue $\lambda=1$:
\begin{equation}\label{eqn:A2.63}
\begin{cases} u''-x^2u+u=\big(\frac{d}{dx}-x\big)u^k, \quad k\geq 2,\\
u(0)=u_0>0
\end{cases}
\end{equation}
 It was shown in \cite{A18} that the solution of (\ref{eqn:A2.63}) is given by
\begin{equation}\label{eqn:A2.68}
u(x)=e^{-\frac{x^2}{2}}\Big{[}\lambda+\sqrt{2k-2}\,{\rm Erfc}\big(\sqrt{\frac{k-1}{2}}x\big)\Big{]}^{\frac{1}{1-k}}
\end{equation}
with
$\lambda=u_0^{1-k}-\sqrt{\frac{\pi(k-1)}{2}}$, where we used the complementary error function defined by
\[
{\rm Erfc}(t)=\int_{t}^{+\infty} e^{-v^2}dv.
\]
Here and in the following,
roots are defined to be
positive for positive
numbers, with continuous
extension to the complex
domain, i.e., we take
principal branches.
Suppose now $\lambda>0$, that is $0<u_0<\big(\frac{\pi(k-1)}{2}\big)^{\frac{1}{2-2k}}$. In this case, since
\[
0<\lambda<\lambda+\sqrt{2k-2}\,{\rm Erfc}\bigg(\sqrt{\frac{k-1}{2}}x\bigg),\]
 the solution $u(x)$ in \eqref{eqn:A2.68} is well defined analytic in $\mathbb{R}$ and
\[
0<u(x)<\lambda^{\frac{1}{1-k}}e^{-\frac{x^2}{2}}.
\]
Similar estimates are valid for $u'(x)$, $u^{\prime\prime}(x)$. Hence  we have $u\in H^2(\mathbb{R})$, and Theorem \ref{mainthm} applies, implying the desired holomorphic  extension $u(z)$ to a sector. However, as observed in \cite{A18}, $u(z)$ is not entire, but has a singularity at $z_0\in \mathbb{C}$ when
\begin{equation}\label{eqn:A2.69}
\lambda+\sqrt{2k-2}\,{\rm Erfc}\bigg(\sqrt{\frac{k-1}{2}}z_0\bigg)=0,
\end{equation}
where ${\rm Erfc}(z)$ is the entire extension of ${\rm Erfc}(x)$.
Such singularities in fact occur, because the great Picard theorem in the complex domain grants the existence of infinitely many solutions $z_0$ of (\ref{eqn:A2.69}) for all $\lambda\in \mathbb{C}$, but for a possible exceptional value, see \cite{Tr}.\par
 Indeed, we now prove the following more precise result.

\begin{proposition}\label{progigante} For every $\lambda>0$, but for a possible exceptional value, and every  $\epsilon>0$, $u(z)$ has a sequence of singularities which tends to infinity in the sector $\pi/4<{\rm arg}\, z<\pi/4+\epsilon$  or in $3\pi/4-\epsilon<{\rm arg}\, z<3\pi/4$.\end{proposition}

\begin{proof} Using the great Picard theorem as above and the reflection properties
\[
{\rm Erfc}(\overline{z})=\overline{{\rm Erfc}({z})},\qquad {\rm Erfc}(-\overline{z})=\sqrt{\pi}-{{\rm Erfc}(\overline{z})}=\sqrt{\pi}-\overline{{\rm Erfc}({z})},
\] which can be verified directly from the definition, it is sufficient to prove that
\begin{equation} \label{7ma}
{\rm Erfc}(z)\to0\ {\rm as}\ z\to\infty\ \textrm{in the sector}\ |{\rm arg}\, z|\leq\pi/4,
\end{equation}
and that,  for every $\epsilon>0$,
 \begin{equation} \label{7ma2}
|{\rm Erfc}(z)|\to+\infty\ {\rm as}\ z\to\infty\ \textrm{in the sector}\ \pi/4+\epsilon<{\rm arg}\,z\leq\pi/2.
\end{equation}
Now, \eqref{7ma} follows at once from the expansion
\[
{\rm Erfc}(z)=\frac{e^{-z^2}}{2z}(1+R(z))\quad {\rm with}\quad |R(z)|\leq \frac{1}{\sqrt{2}|z|^2},\
\]
valid when $|{\rm arg}\,z|\leq\pi/4$; see e.g. \cite[pages 18--20]{lebedev}.\par
The property \eqref{7ma2} can be verified directly as follows. Observe that, for $z=x+iy$, $x>0$, $y>0$, we can write
\[
{\rm Erfc}(z)=-\int_\gamma e^{-u^2}du,
\]
where the path $\gamma$ is given by the hyperbola through $z=x+iy$ with parametrization $u=\gamma(t)=\frac{xy}{t}+it$, $t\in(0,y]$.  Then
\begin{align*}
|{\rm Erfc}(x+iy)|&=\big|\int_0^y e^{-\frac{x^2y^2}{t^2}+t^2}\left(-\frac{xy}{t^2}+i\right)\,dt\Big|\\
&\geq \int_0^y  e^{-\frac{x^2y^2}{t^2}+t^2}\,dt.
\end{align*}
Let $0<\mu<1$ be a number to be chosen later. We have
\[
|{\rm Erfc}(x+iy)|\geq \int_{\mu y}^y  e^{-\frac{x^2y^2}{t^2}+t^2}\,dt\geq (1-\mu)ye^{-\mu^{-2}x^2+\mu^2 y^2}.
\]
Now, if $z$ belongs in addition to the sector in \eqref{7ma2}, we have $0<x<\tilde{\epsilon}y$, for some $\tilde{\epsilon}<1$.  We obtain then
\[
|{\rm Erfc}(x+iy)|\geq (1-\mu)y e^{(\mu^2-\tilde{\epsilon}^2\mu^{-2})y^2}.
\]
If we choose $\mu>\sqrt{\tilde{\epsilon}}$, we get $|{\rm Erfc}(x+iy)|\to+\infty$ as $y\to+\infty$, which gives the desired conclusion when $x={\rm Re}\, z>0$. The case when $x=0$ is immediate, because
\[
 {\rm Erfc}(iy)=-\int _0^y e^{t^2}\,dt+\frac{\sqrt{\pi}}{2}.
 \]
Property \eqref{7ma2} is then proved. \end{proof}
\section*{Acknowledgements}
We are very indebted to Giacomo Gigante for highlighting correspondence on the solutions of the equation \eqref{eqn:A2.69}, which led to Proposition \ref{progigante}.

\end{document}